\documentclass[12pt]{amsart}

\usepackage{amsmath,amssymb,amsbsy,amsthm,amsfonts}

\usepackage{slashbox}

\def \endproof {\quad \hfill  \rule{2mm}{2mm} \par\medskip}

\DeclareMathSymbol{\subsetneqq}{\mathbin}{AMSb}{36}

\newcommand{\R}{\mathbb{R}}
\newcommand{\N}{\mathbb{N}}

\newcommand{\Z}{\mathbb{Z}}

\newcommand{\dint}{\displaystyle\int}
\newcommand{\dsum}{\displaystyle\sum}
\newcommand{\dsup}{\displaystyle\sup}
\newcommand{\dlim}{\displaystyle\lim}

\newcommand{\dliminf}{\displaystyle\liminf}

\newtheorem{th1}{{\bf Theorem}}[section]
\newtheorem{thm}[th1]{{\bf Theorem}}
\newtheorem{lem}[th1]{{\bf Lemma}}
\newtheorem{prop}[th1]{{\bf Proposition}}
\newtheorem{cor}[th1]{{\bf Corollary}}

\theoremstyle{remark}
\newtheorem{rem}[th1]{\bf Remark}

\theoremstyle{definition}
\newtheorem{defi}[th1]{\bf Definition}

\author[S.~Ibrahim]{Slim Ibrahim}
\address{Department of Mathematics and Statistics,\\University of Victoria\\
 PO Box 3060 STN CSC\\   Victoria, BC, V8P 5C3\\ Canada}
\email{\sl ibrahim@math.uvic.ca}
\urladdr{ http://www.math.uvic.ca/~ibrahim/}
\thanks{}

\author{Mohamed Majdoub}
\address{University Tunis ElManar,
Faculty of Sciences of Tunis, Department of Mathematics.}
\email{\sl mohamed.majdoub@fst.rnu.tn}
\thanks{M. M. is grateful to the Laboratory of
PDE and Applications at the Faculty of Sciences of Tunis.}

\author{Nader Masmoudi}
\address{New York University \\
The Courant Institute for Mathematical Sciences, USA.}
\email{\sl masmoudi@courant.nyu.edu}
\thanks{N. M is partially supported by an NSF Grant DMS-0703145}

\title[Well and ill-posedness issues ]
{Well and ill-posedness issues for energy supercritical waves}
\date{\today}

\begin{document}
\begin{abstract}
 We investigate the initial value problem for some energy supercritical semilinear
wave equations. We establish local existence in suitable spaces with continuous flow. We also obtain some {\it ill-posedness/weak ill-posedness}
 results. The proof uses the finite speed of propagation and a quantitative study of the associated ODE. It does not require any scaling invariance of the equation.
\end{abstract}



\maketitle


\section{Introduction}
In this work, we discuss some well-posedness issues of the Cauchy
problem associated to the semilinear wave equation
\begin{equation}
\label{wave1}
\partial_t^2 u-\Delta u+F'(u)=0,\quad\mbox{in}\quad \R_t\times\R_x^d,
\end{equation}
where $d\geq 2$ and $F:\R\longrightarrow\R$ is an {\it even} regular
function satisfying
\begin{equation}
\label{general}
F(0)=F'(0)=0\quad\mbox{and}\quad u\;F'(u)\geq 0.
\end{equation}
The above assumption on $F$ include the massive case {\rm i.e.} the Klein-Gordon equation. With hypothesis \eqref{general}, one can construct a global weak solution with finite energy data using a standard compactness
argument (see, for example \cite{Strauss}). However, the construction of strong solutions (even local) requires  some control on the growth at infinity and more tools. As regards the growth of the nonlinearity $F$, we distinguish two cases. For dimension $d\geq 3$ we shall assume that our Cauchy problem is $H^1$-supercritical in the sense that
\begin{equation}
\label{3+}\dfrac{F(u)}{|u|^{\frac{2d}{d-2}}}\nearrow \,  +\infty, \quad u \to \infty\,.
\end{equation}
In two space dimensions and thanks to Sobolev embedding, any Cauchy problem with polynomially growing nonlinearities is locally well posed regardless of the sign of the nonlinearity and the growth of $F$ at infinity. This is a limit case of \eqref{3+}. Square exponential nonlinearities were investigated first in \cite{NO1} where the authors showed global existence and scattering for small Cauchy data, then in \cite{A} where local existence was obtained under restrictive conditions, and finally in \cite{IMM} where a new notion of criticality based on the size of the energy appears. In this paper, we examine the situation of other growths of exponential nonlinearities (not necessarily square). More precisely, when $d=2$, we assume either
\begin{equation}
\label{2Dinfty}
\dfrac{\log(F(u))}{|u|^{2}}\nearrow\, +\infty, \quad u \to \infty,
\end{equation}
or
\begin{equation}
\label{expgrow}
\exists\;0<q\leq 2\; \quad\mbox{ \sf s.t.}
\quad\dfrac{\log(F(u))}{|u|^{q}}=O(1), \,\quad u \to \infty.
\end{equation}
The model example that we are going to work with when $d=3$ is given by
\begin{equation}
\label{NL3+}
\partial_t^2 u - \Delta u +u^7=0.
\end{equation}
It is a good prototype for all higher dimensions $d\geq3$ illustrating assumption \eqref{3+}. In the case $d=2$, we take
\begin{equation}
\label{Wave2D}
\partial_t^2 u - \Delta u +u\,(1+u^2)^{\frac{q-2}{2}}\;{\rm e}^{4\pi \left((1+u^2)^{\frac{q}{2}}-1\right)}=0,
\end{equation}
with $q>0$ illustrating either the cases \eqref{2Dinfty} or \eqref{expgrow}, depending upon the fact that $q>2$ or $q\leq 2$, respectively.\\

For any weak solution of (\ref{wave1}), define the total energy by
$$
E(u(t)){\buildrel\hbox{\footnotesize def}\over
=}\;\|\nabla_{t,x}u(t)\|_{L^2_x}^2\;+\; \dint_{\R^d}\;2F(u(t))\;dx.
$$

The energy of the data $(\varphi,\psi)\in\dot{H}^1\times L^2$ is given
by
$$
E(\varphi,\psi){\buildrel\hbox{\footnotesize
def}\over
=}\;\|\nabla\varphi\|_{L^2_x}^2\;+\;\|\psi\|_{L^2_x}^2\;+\;
\dint_{\R^d}\;2F(\varphi)\;dx.
$$
When $\psi=0$, we abbreviate the notation $E(\varphi,0)$ to simply $E(\varphi)$.

In the sequel, we adopt the following classical definition of local/global well-posedness.
\begin{defi}
\label{d1} Let ${\mathbf X}$ be a Banach space\footnote{Typically, ${\mathbf X}={\mathbf X}^s:=B^s_{p,q}\times B_{p,q}^{s-1}$, for some suitable choice of $s$, $p$ and $q$}.
\begin{itemize}
 \item
The Cauchy problem associated to \eqref{wave1} is {\it locally well-posed} in ${\mathbf X}$, abbreviated as LWP,
if for every data $(u_0,u_1)\in {\mathbf X}$, there exists a time $T>0$ and a unique\footnote{In some cases the uniqueness holds in more restrictive space.} (distributional) solution $u:[-T,T]\times\R^d\longrightarrow\R$ to
\eqref{wave1} such that $(u,\partial_tu)\in{\mathcal C}([-T,T]; {\mathbf X})$, $(u,\partial_t u)(t=0)=(u_0,u_1)$, and such that the solution map $(u_0,u_1)\longmapsto (u,\partial_tu)$ is continuous from ${\mathbf X}$ to ${\mathcal C}([-T,T]; {\mathbf X})$.
\item The Cauchy problem is {\it globally well-posed} (GWP) if the time $T$ can be taken arbitrary.
\item The Cauchy problem is {\it strongly well-posed} (SWP) if the solution map is uniformly continuous.
\item The Cauchy problem is {\it ill-posed} (IP) if the solution map is not continuous.
\item
The Cauchy problem is said {\it weakly ill-posed} on a set ${\mathbf Y}\subset {\mathbf X}$ (WIP), if the solution map
$$
(u_0,u_1)\in {\mathbf Y}\longmapsto (u,\partial_tu)
$$
is not uniformly continuous.
\end{itemize}
\end{defi}
Let us recall a few historic facts about this  problem. First, when the space dimension $d\geq 3$, the defocusing semilinear wave
equation with power $p$ reads
\begin{equation}
\label{N} \partial_t^2u-\Delta u+|u|^{p-1}u=0,
\end{equation}
where $p>1$. This problem has been widely
investigated and there is a large literature dealing with the
well-posedness theory of (\ref{N}) in the scale of the Sobolev
spaces $H^s$. Second, for the global solvability in the energy space
$\dot{H}^1\times L^2$, there are mainly three cases. The first case
is when $p<p_c$ where $p_c=\frac{d+2}{d-2}$, this is the subcritical
case. In this case, Ginibre and Velo \cite{GV1} showed that the
problem (\ref{N}) is globally well-posed in the energy space. If the
exponent $p$ is critical (which means $p=p_c$) this problem was
 solved by Shatah-Struwe (\cite{SS2} and references therein). Finally in the case
$p>p_c$, the well-posedness in the energy space is an open problem
except for some partial results about weak illposedness. See for example \cite{CCT},
\cite{Le}, \cite{Leb}, \cite{BGT} and \cite{BIG}. See also \cite{tao} for a result about global regularity for a logarithmically energy-supercritical wave equation in the radial case.\\ In dimension two, $H^1$-critical nonlinearities seem to be of
exponential type{\footnote{In fact, the critical nonlinearity is of
exponential type in any dimension $d$ with respect to $H^{d/2}$
norm.}}, since every power is $H^1$-subcritical. In a recent work \cite{IMM1}, the case $F(u)=\frac{1}{8\pi}\left({\rm e}^{4\pi u^2}-1\right)$ was
investigated and the criticality was proposed with respect to the
size of the energy. Moreover, the local strong well-posedness was shown under the size restriction $\|\nabla u_0\|_{L^2}<1$. In this paper, we want to investigate the local wellposedness regardless of the size of the initial data.\\

The ill posedness results of \cite{CCT} are based on the scaling invariances of the wave and Shr\"odinger equations with homogeneous nonlinearities. The idea is to approximate the solution by its corresponding ODE (at the zero dispersion limit). Since solutions of the ODE are periodic in time, then a de-coherence phenomena occurs for small time since the ODE solutions oscillate fast. Our idea to overcome the absence of scaling invariance is to choose one step-functions as initial data (i.e. functions constant near zero). The presence of the step immediately guarantees the equality between the PDE and the ODE solutions in a backward light cone, thanks to the finite speed of propagation. The length of the step can be adjusted (in the supercritical regime) so that ill-posedness/weak ill-posedness occurs inside the light cone.\\

This paper is organized as follows.
In Section 2, we state our main results.
In Section 3, we recall some basic definitions and auxiliary lemmas.
In Sections 4, we investigate the energy regularity regime.
Section 5 is devoted to the low regularity data.\\

Finally, we mention that, $C$ will
be used to denote a constant which may vary from line to line.
We also use $A\lesssim B$ to denote an estimate of the form $A\leq C B$
for some constant $C$.

\section{Main results}

\subsection{Energy regularity data}
 First we show that if the general assumptions \eqref{general}-\eqref{3+} (or \eqref{general}-\eqref{2Dinfty}) are satisfied, then the nonlinearity is too strong to ensure the local well-posedness in the energy space. Hence we have
\begin{thm}
\label{Mainresult1}
Assume that $d\geq 3$ and \eqref{general}, \eqref{3+} or $d=2$ and \eqref{general}, \eqref{2Dinfty}. Then
\label{energy}
\begin{enumerate}
 \item[1)] There exists a sequence $(\varphi_k)$ in $\dot{H}^1$ and
a sequence $(t_k)$ in $(0,1)$ satisfying
$$
\|\nabla\varphi_k\|_{L^2_x}\longrightarrow 0,\quad
t_k\longrightarrow 0,\quad \dsup_k\;E(\varphi_k)<\infty,
$$
and such that any weak solution $u_k$ of (\ref{NL3+}) with initial
data $(\varphi_k,0)$ satisfies
$$
\dliminf_{k\rightarrow+\infty}\;\|\partial_t
u_k(t_k)\|_{L^2_x}\gtrsim 1.
$$
In particular the Cauchy problem is ill-posed in $H^1\times L^2$.
\item[2)] If we relax the condition $\dsup_k\;E(\varphi_k)<\infty $ by taking $\lim_k\dint F(\varphi_k)=+\infty$, we can even get
$$
\dlim_k\|\partial_t u_k(t_k)\|_{L^2_x} = \infty\,.
$$
\end{enumerate}
\end{thm}
\begin{rem}
In \cite{tao}, Tao has shown the global well-posedness of the logarithmic energy supercritical wave equation in $H^{1+\varepsilon}\times H^{\varepsilon}$ for any $\varepsilon>0$. The above Theorem shows that $\varepsilon$ cannot be taken zero.
\end{rem}
The above Theorem covers model \eqref{Wave2D}  in two space dimensions with $q>2$. When $q<2$, recall that the global well-posedness in the energy space can easily be obtained through the sharp Moser-Trudinger inequality combined with the following simple observation
$$
\forall\,\varepsilon>0,\quad \exists\,C_\varepsilon>0\quad\mbox{s.t.}\quad
\forall\, u\in\R,\qquad\Big|(1+u^2)^{\frac{q-2}{2}}\;{\rm e}^{4\pi (1+u^2)^{\frac{q}{2}}}-{\rm e}^{4\pi}\Big|\leq\,C_\varepsilon\,\left({\rm e}^{\varepsilon u^2}-1\right).
$$
In the case $q=2$, the local well-posedness for the Cauchy problem associated to \eqref{Wave2D} in the energy space was first established in \cite{NO1, NO2} for small Cauchy data. Later on, optimal smallness for well-posedness was investigated, first in \cite{A} for radially symmetric initial data $(0,u_1)$, and then in \cite{IMM1, IMM2} for general data. The following result generalizes the previous results to any data in the energy space regardless of its size.
\begin{thm}
\label{loc-large} Let $(u_0,u_1)\in H^{1}\times L^{2}$. There exists a time $T>0$ and a unique solution $u$ of \eqref{Wave2D}  with $q=2$  in the space $C_T(H^{1})\cap C^1_T(L^{2})$ satisfying $u(0,x)=u_0(x)$ and $\dot u(0,x)=u_1(x)$. Moreover, the solution map is continuous on $H^1\times L^2$.
\end{thm}

In \cite{IMM2} it is shown that the local solutions of \eqref{Wave2D} (with $q=2$) are global whenever the total energy $E\leq1$, where
$$
E(u(t)){\buildrel\hbox{\footnotesize def}\over
=}\;\|\nabla_{t,x}u(t)\|_{L^2_x}^2\;+\frac{1}{4\pi}\; \int_{\R^2}\;{\rm e}^{4\pi u^2}-1\;dx.
$$
Indeed, in that case, the Cauchy problem is strongly well posed. The following result shows the weak ill-posedness on the set $\{\,E<1+\delta\,\}$ for any $\delta>0$ . More precisely
\begin{thm}
\label{2Dsupercritical} Let $\nu>0$. There exist a sequence of positive real
numbers $(t_k)$ tending to zero and two sequences $(u_k)$ and
$(v_k)$ of solutions of the nonlinear Klein-Gordon equation
\begin{equation}
\label{nlkg} \Box u+u {\rm e}^{4\pi u^2}=0
\end{equation}
satisfying the following:
$$
 \|(u_k-v_k)(t=0,\cdot)\|_{H^1}^2+\|
 \partial_t(u_k-v_k)(t=0,\cdot)\|_{L^2}^2=
 \circ(1)\;\hbox{ as }\;k\rightarrow+\infty,
$$

$$
0< E(u^k,0)-1\leq {\rm e}^3\nu^2 ,\quad 0< E(v^k,0)-1\leq \nu^2,
$$
 and
$$
 \liminf_{k\longrightarrow\infty}\|
 \partial_t(u_k-v_k)(t_k,\cdot)\|^2_{L^2}\geq\frac{\pi}{4}({\rm e}^2+{\rm e}^{3-8\pi})\nu^2.
$$
\end{thm}

Notice that Theorem \ref{loc-large} yields the continuity with respect to the initial data and Theorem \ref{2Dsupercritical} yields that there is no uniform continuity if the energy is larger than $1$ (supercritical regime).
\begin{rem}
Very recently, Struwe \cite{Struwe09} has constructed global smooth solutions for the 2D energy critical wave equation with radially symmetric data. Although the techniques are different, this result might be seen as an analogue of Tao's result \cite{tao} for the 3D energy supercritical wave equation. Our Theorem \ref{2Dsupercritical} shows just the weak ill-posedness in the supercritical case. This is weaker  than the result in higher dimensions where the flow fails to be continuous at zero as shown in \cite{CCT}. The reason behind this is that small date are always subcritical in the exponential case.
\end{rem}
\subsection{Low regularity data for the model \eqref{Wave2D}}\quad\\
Now that the local well/ill-posedness is clarified in the energy space for dimension $d\geq2$, our next task in this paper is to seek for the ``largest possible spaces'' in which we have local well-posedness for the Cauchy problem associated to the model \eqref{Wave2D}. Recall that we have the embeddings
\begin{equation}
 H^1(\R^2)\hookrightarrow B^1_{2,\infty}(\R^2)\hookrightarrow H^s(\R^2),\quad s<1\,.
\end{equation}
The next theorem show the failure of the well posedness in spaces slightly bigger than the energy space in the case $q=2$. This means that the Cauchy problem posed either in $B_{2,\infty}^1$ or  $H^s$ with $s<1$ are supercritical at those regularity level. More specifically
\begin{thm}
 \label{IP-2D}
Assume that $q=2$. Let ${\mathcal W}:=\Big\{\, u\in L^2\quad\mbox{s.t.}\quad
\nabla u\in L^{2,\infty}\,\Big\}$ where $L^{2,\infty}$ is the classical Lorentz space\footnote{It is defined by its norm $\|u\|_{L^{2,\infty}}:=\dsup_{\sigma>0}\Big(\sigma\;\mbox{meas}^{1/2}\{\,|u(x)|>\sigma\}\Big)$.}. Then
\begin{enumerate}
\item[1)] There exists a sequence $(\varphi_k)$ in ${\mathcal
W}$ and a sequence $(t_k)$ in $(0,1)$ satisfying
$$
\|\varphi_k\|_{{\mathcal W}}\longrightarrow 0,\quad
t_k\longrightarrow 0,
$$
and such that any weak solution $u_k$ of (\ref{Wave2D}) with initial
data $(\varphi_k,0)$ satisfies
$$
\dlim_{k\rightarrow\infty}\;\|\partial_t u_k(t_k)\|_{L^{2,\infty}}=\infty.
$$
 \item[2)] There exists a sequence $(\varphi_k)$ in ${\mathcal
B}^1_{2,\infty}$ and a sequence $(t_k)$ in $(0,1)$ satisfying
$$
\|\varphi_k\|_{{\mathcal B}^1_{2,\infty}}\longrightarrow 0,\quad
t_k\longrightarrow 0,
$$
and such that any weak solution $u_k$ of (\ref{Wave2D}) with initial
data $(\varphi_k,0)$ satisfies
$$
\dlim_{k\rightarrow\infty}\;\|\partial_t u_k(t_k)\|_{{\mathcal
B}^0_{2,\infty}}=\infty.
$$
In particular, the flow fails to be continuous at $0$ in the ${\mathcal W}\times L^{2,\infty}$ topology or
${\mathcal B}^1_{2,\infty}\times {\mathcal B}^0_{2,\infty}$
topology.
\item[3)] Let $s<1$. There exists a sequence $(\varphi_k)$ in
$H^s$ and a sequence $(t_k)$ in $(0,1)$ satisfying
$$
\|\varphi_k\|_{H^s}\longrightarrow 0,\quad t_k\longrightarrow 0,
$$
and such that any weak solution $u_k$ of (\ref{Wave2D}) with initial
data $(\varphi_k,0)$ satisfies
$$
\dlim_{k\rightarrow\infty}\;\|\partial_t
u_k(t_k)\|_{H^{s-1}}=\infty.
$$
In particular, the flow fails to be continuous at $0$ in the
$H^s\times H^{s-1}$ topology.
\end{enumerate}
\end{thm}

This theorem can be seen as a consequence of the following general result about arbitrary $1\leq q<\infty$. Indeed, equation \eqref{Wave2D} is subcritical at the regularity of the Besov space ${\mathcal B}^1_{2,q'}$\;\footnote{As usually, $q'$ denotes the Lebesgue conjugate exponent of $q$.} but supercritical at the $H^s$ regularity level with $s<1$. More precisely
\begin{thm}
\label{WP-IP}
Assume that $1\leq q<\infty$.
\begin{enumerate}
 \item[1)] Let $(u_0,u_1)\in {\mathcal
B}^1_{2,q'}\times{\mathcal B}^0_{2,q'}$\footnote{As we will see in the proof, when $q'=\infty$ the appropriate space is $\tilde{\mathcal B}^1_{2,\infty}$, the closure of smooth compactly supported function in the usual Besov space ${\mathcal
B}^1_{2,\infty}$.}. There exists a time
$T>0$ and a unique solution $u$ of (\ref{Wave2D}) with initial data
$(u_0,u_1)$ in the space $C_T({\mathcal B}^1_{2,q'})\cap
C^1_T({\mathcal B}^0_{2,q'})$.
\item[2)] Let $s<1$. There exists a sequence $(\varphi_k)$ in
$H^s$ and a sequence $(t_k)$ in $(0,1)$ satisfying
$$
\|\varphi_k\|_{H^s}\longrightarrow 0,\quad t_k\longrightarrow 0,
$$
and such that any weak solution $u_k$ of (\ref{Wave2D}) with initial
data $(\varphi_k,0)$ satisfies
$$
\dlim_{k\rightarrow+\infty}\;\|\partial_t
u_k(t_k)\|_{H^{s-1}}=\infty.
$$
In particular, the flow fails to be continuous at $0$ in the
$H^s\times H^{s-1}$ topology.
\end{enumerate}
\end{thm}

\begin{rem}
\label{Schrod1} The same well-posedness results can be derived for the
corresponding two dimensional nonlinear Schr\"odinger equations.
\end{rem}
We end this section with the following board which clarify the picture of well/ill-posedness.
\begin{table}[!ht]
\tiny{\begin{tabular}{|c|c|c|c|c|c|c|}
\hline
\backslashbox[2cm]{Data's regularity}{Setting} &  {\tiny$d\geq 3$ and \eqref{3+}}&{\tiny$d=2$ and \eqref{2Dinfty}} &{\tiny $d=2$ and $q<2$}&{\tiny$d=q=2$ and $E> 1$}&{\tiny$d=q=2$ and $E\leq 1$}\\
\hline
\tiny{$H^1$} & \tiny{WIP} &\tiny{IP}&  \tiny{GWP}  \&  \tiny{SWP}& \tiny{LWP}  \&  \tiny{WIP}&\tiny{GWP}  \&  \tiny{SWP}\\
\hline
\tiny{$\mathcal{B}^1_{2,\infty}$} & \tiny{IP} & \tiny{IP} &\tiny{LWP}&\tiny{IP}&\tiny{IP}\\
\hline
\tiny{$H^s$ with $s<1$} &\tiny{IP}&\tiny{IP}&\tiny{IP}&\tiny{IP}&\tiny{IP}\\
\hline
\end{tabular}}
\end{table}

\section{Background Material}

In this section we will fix the notation, state the basic definitions and recall some known and useful tools.


\subsection{Besov spaces}
For the convenience of the reader, we recall the definition and some properties of Besov spaces.
\begin{defi}
\label{homo-bes}
Let $\chi$ be a function in $\mathcal{S}(\R^d)$ such that ${\chi}(\xi)=1$ for $|\xi|\leq 1$ and $\widehat{\phi}(\xi)=0$ for $|\xi|> 2$. Define the function $\psi(\xi)={\chi}(\xi/2)-{\chi}(\xi)$. Then the (homogeneous) frequency localization operators are defined by
$$
\widehat{\dot{\triangle}_j\,u}(\xi)=\psi(2^{-j}\xi)\widehat{u}(\xi)\,\mbox{ for all }j\in\Z\;.
$$
If $s<\frac{d}{p}$, then $u$ belongs to the homogenous Besov space $\dot{\mathcal B}^s_{p,q}(\R^d)$ if and only if the partial sum $\sum_{-m}^m\,\dot{\triangle}_j\,u$ converges to $u$ as a tempered distribution and the sequence $\left(2^{sj}\,\|\triangle_j\,u\|_{L^p}\right)$ belongs to $\ell^q(\Z)$.
\end{defi}

To define the inhomogeneous Besov spaces, we need an inhomogeneous frequency localization.
\begin{defi}
 \label{inhom-bes}
The inhomogeneous frequency localization operators are defined by
\begin{eqnarray*}
 \widehat{\triangle_j u}(\xi)&=&\; \left\{
\begin{array}{cllll}0 \quad&\mbox{if}&\quad
j\leq -2,\\\\
\chi(\xi)\widehat{u}(\xi) \quad&\mbox{if}&\quad j=-1
,\\\\
\psi(2^{-j}\xi)\,\widehat{u}(\xi)\quad&\mbox{if}&\quad j\geq 0.
\end{array}
\right.
\end{eqnarray*}
For $N\in\N$, set $S_N=\dsum_{j\leq N-1}\;\triangle_j$.
We say that $u$ belongs to the inhomogeneous Besov space ${\mathcal B}^s_{p,q}(\R^d)$ if and only if $u\in{\mathcal S}'$ and $\|u\|_{{\mathcal B}^s_{p,q}}<\infty$ where
\begin{eqnarray*}
\|u\|_{{\mathcal B}^s_{p,q}}&=&\; \left\{
\begin{array}{cllll} \|\triangle_{-1}u\|_{L^p}+\Big(\dsum_{j=0}^\infty\,2^{jqs}\,\|\triangle_{j}u\|_{L^p}^q\Big)^{1/q}\quad&\mbox{if}&\quad
q<\infty,\\\\
\|\triangle_{-1}u\|_{L^p}+\dsup_{j\geq 0}\,2^{js}\,\|\triangle_{j}u\|_{L^p} \quad&\mbox{if}&\quad q=\infty\,.
\end{array}
\right.
\end{eqnarray*}
\end{defi}

We recall without proof the following properties of the operators $\triangle_j$ and Besov spaces
(see \cite{RS}, \cite{Tr1}, \cite{Tr2} and \cite{Tr3}).
\begin{itemize}
\item {\sf Bernstein's inequality.} For all $1\leq p\leq q\leq\infty$ we have
$$
\|\triangle_j\,u\|_{L^q(\R^d)}\,\leq\,C\,2^{jd(\frac{1}{p}-\frac{1}{q})}\,\|\triangle_j\,u\|_{L^p(\R^d)}\,.
$$
 \item {\sf Embeddings.}
\begin{equation}
 \label{embedd}
{\mathcal B}_{p,q}^s(\mathbb{R}^d)\hookrightarrow{\mathcal
B}_{p_1,q_1}^{s_1}(\mathbb{R}^d)
\end{equation}
 whenever $$ s-\frac{d}{p}\geq s_1-\frac{d}{p_1},\quad 1\leq p\leq p_1\leq\infty,\quad 1\leq q\leq q_1\leq\infty,\quad s,s_1\in\mathbb{R}\,.
 $$
\item {\sf Equivalent norm.} For $s>0$ we have
\begin{equation}
 \label{normBes}
\|u\|_{{\mathcal B}^s_{p,q}}\,\approx\,\|u\|_{L^p}+   \|\nabla u\|_{\dot{\mathcal B}^{s-1}_{p,q}}\,.
\end{equation}
\item Sobolev spaces and H\"older spaces are special cases of Besov spaces, that is $H^s={\mathcal B}^s_{2,2}$ and $C^\sigma={\mathcal B}^\sigma_{\infty,\infty}$ for non-integer $\sigma>0$.
\end{itemize}

We shall also use the following result of Runst and Sickel \cite{RS} about functions which operate by pointwise multiplication in Besov spaces.
\begin{thm}[{\cite{RS} Theorem 4.6.2}]
\label{Bourdaud}
Let $|s|<d/2$. Any function belonging to $\dot{\mathcal B}^{d/2}_{2,\infty}\cap L^\infty(\R^d)$ is a pointwise multiplier in the Besov space $\dot{\mathcal B}^{s}_{2,q}(\R^d)$.
\end{thm}
An important application of this Theorem\footnote{We are grateful to {\sf G\'erard Bourdaud} for providing us this reference and a proof of the application.} which will be used in the sequel is the fact that the function $f(x):=\frac{x}{r}$ operates on $\dot{\mathcal B}^0_{2,\infty}(\R^2)$ via pointwise multiplication. Indeed, according to Theorem \ref{Bourdaud} it suffices to show that $f$ belongs to $\dot{\mathcal B}^{1}_{2,\infty}({\mathbb R}^2)$. For this, note that $\widehat{f}$ is an homogenous distribution of degree $-2$, belonging to the $C^\infty$ class outside the origin. We can then define $g\in {\mathcal S}$ by $\widehat{g} = \psi \widehat{f}$. Hence $\triangle_j\,f (x)= g(2^jx)$ and $\|\triangle_j\,f\|_{L^2} = 2^{-j} \|g\|_{L^2}$.
\subsection{2D Strichartz estimate and Logarithmic inequality}
Recall the following 2D Strichartz estimate.
\begin{prop}
 \label{Str}
\begin{equation}
\label{Stz}
\|u\|_{L^4((0,T);B^{1/4}_{\infty,2})} \lesssim \|\partial_t^2 u -\Delta u + u\|_{L^1((0,T);L^2)}
 + \|u(0)\|_{H^1} + \|\partial_t u(0)\|_{L^2}.
\end{equation}
\end{prop}
Remark that using the embedding \eqref{embedd}, we can replace $B^{1/4}_{\infty,2}$ with the H\"older space $C^{1/4}$.\\

The following lemma shows that we can estimate the $L^\infty$ norm
by a stronger norm but with a weaker growth (namely logarithmic).
\begin{lem}
\label{logBeso}
Let $0<\alpha<1$ and $1\leq q\leq\infty$. There exists a constant $C$ such that
\begin{equation}
\label{logbeso}
\|u\|_{L^\infty}\;\leq\; C\|u\|_{{\mathcal B}^{1}_{2,q'}}\,\log^{\frac{1}{q}}\Big({\rm e}+\frac{\|u\|_{\mathcal{C}^\alpha}}{\|u\|_{{\mathcal B}^{1}_{2,q'}}}\Big)\,.
\end{equation}
\end{lem}
Note that similar inequalities appeared in Br\'ezis-Gallouet \cite{BrGa} and has been improved (with respect to the best constant) in \cite{IMM} in the following sense.
\begin{lem}[\cite{IMM}, Theorem 1.3]
\label{Hmu}
 Let $0<\alpha<1$.  For any $\lambda>\frac{1}{2\pi\alpha}$ and
any $0<\mu\leq1$, a constant $C_{\lambda}>0$ exists such that, for
any function $u\in H^1(\R^2)\cap{\mathcal C}^\alpha(\R^2)$
\begin{equation}
\label{H-mu} \|u\|^2_{L^\infty}\leq
\lambda\|u\|_{H_\mu}^2\log\left(C_{\lambda} +
\frac{8^\alpha\mu^{-\alpha}\|u\|_{{\mathcal
C}^{\alpha}}}{\|u\|_{H_\mu}}\,\,\,\right),
\end{equation}
where $H_\mu$ is defined by the norm $\|u\|_{H_\mu}^2:=\|\nabla u\|_{L^2}^2+\mu^2\|u\|_{L^2}^2$.
\end{lem}

\begin{proof}[Proof of Lemma \ref{logBeso}] Write
$$
u=\dsum_{j=-1}^{N-1}\,\triangle_j\,u+\dsum_{j=N}^\infty\,\triangle_j\,u,
$$
where $N$ is a nonnegative integer which will be chosen later. Using Bernstein's inequality, we get
\begin{eqnarray*}
\|u\|_{L^\infty}&\leq&C\dsum_{j=-1}^{N-1}\,2^j\,\|\triangle_j\,u\|_{L^2}+
\dsum_{j=N}^\infty\,2^{-j\alpha}\left(2^{j\alpha}\,\|\triangle_j\,u\|_{L^\infty}\right)\\
&\leq&C\,\Big(N^{1/q}\,\|u\|_{{\mathcal B}^{1}_{2,q'}}+\frac{2^{-N\alpha}}{1-2^{-\alpha}}\,\|u\|_{\mathcal{C}^\alpha}\Big)\,.
\end{eqnarray*}
Choose
$$
N\sim\,\frac{1}{\alpha\log 2}\log\Big({\rm e}+\frac{\|u\|_{\mathcal{C}^\alpha}}{\|u\|_{{\mathcal B}^{1}_{2,q'}}}\Big),
$$
we obtain \eqref{logbeso} as desired.
\end{proof}

\subsection{Oscillating second order ODE}
Here we recall a classical result about ordinary differential equations.
\begin{lem}
\label{l-ode} Let $F:\R\longrightarrow\R$ be a smooth function and
consider the following ODE
\begin{equation}
\label{ode1} \ddot{x}(t)+F'(x(t))=0,\quad (x(0),
\dot{x}(0))=(x_0,0),\quad x_0>0.
\end{equation}
The equation (\ref{ode1}) has a periodic non constant solution if
and only if the function $G: y\longmapsto 2\Big(F(x_0)-F(y)\Big)$
has two simples distincts zeros $\alpha$ and $\beta$ with
$\alpha\leq x_0\leq \beta$ and such that $G$ has no zero in the
interval $]\alpha,\beta[$. In this case, the period is given by
$$
T=2\;\dint_{\alpha}^{\beta}\;\dfrac{dy}{\sqrt{G(y)}}=\sqrt{2}\dint_{\alpha}^{\beta}\;\dfrac{dy}{\sqrt{F(x_0)-F(y)}}.
$$
\end{lem}

\subsection{Moser-Trudinger inequalities}
It is known that the Sobolev space $H^1(\R^2)$ is embedded in all Lebesgue spaces $L^p$ for $2\leq p<\infty$ but not in $L^\infty$. Moreover, $H^1$ functions are in the so-called Orlicz space i.e. their exponentials are integrable for every growth less than ${\rm e}^{u^2}$. Precisely, we have the following Moser-Trudinger inequality (see \cite{Tan, Ru} and references therein).
\begin{prop}
\label{p} Let $\alpha\in (0,4\pi)$. A constant $c_\alpha$ exists
such that
\begin{equation}
\label{e5} \int_{\R^2}\,\left({\rm e}^{\alpha|u(x)|^2}-1\right)\,dx\leq c_\alpha \|u\|_{L^2}^2
\end{equation}
for all $u$ in $H^1(\R^2)$ such that $\|\nabla
u\|_{L^2(\R^2)}\leq1$. Moreover, if $\alpha\geq 4\pi$, then
(\ref{e5}) is false.
\end{prop}
 We point out that $\alpha=4\pi$ becomes admissible in
(\ref{e5}) if we require $\|u\|_{H^1(\R^2)}\leq1$ rather than
$\|\nabla u\|_{L^2(\R^2)}\leq1$. Precisely, we have
$$\sup_{\|u\|_{H^1(\R^2)}\leq
1}\;\;\dint_{\R^2}\,\left({\rm e}^{4\pi|u(x)|^2}-1\right)\,dx<\infty$$ and this is false for
$\alpha>4\pi$. See \cite{Ru} for more details.\\
The above estimates obviously control any exponential power with smaller growth (i.e. $q<2$). However no estimate holds if the growth is higher (i.e $q>2$). Hence, the value $q=2$ is also another criticality threshold for problems involving such nonlinearities.

\subsection{Some technical lemmas}

\begin{lem}
\label{lem-bound} For any $0<a<1$, denote by
$$
I(a)=\dint_a^1\;r\;{\rm e}^{4 a^2\log^2 r}\;dr.
$$
Then, we have
\begin{equation}
\label{bound}
 I(a)\leq 2.
\end{equation}
\end{lem}
\begin{proof}[Proof of Lemma \ref{lem-bound}]
Changing the variable $s=-2a\log r$ yields
\begin{eqnarray*}
I(a)&=&\frac{1}{2a}\;{\rm e}^{-\frac{1}{4 a^2}}\;\dint_0^{-2a\log
a}\;{\rm e}^{(s-\frac{1}{2a})^2}\;ds\\
&\leq&2A\;{\rm e}^{-A^2}\;\dint_0^{A}\;{\rm e}^{y^2}\;dy,
\end{eqnarray*}
where $A=\frac{1}{2a}$. Now, using the following estimate true for
all nonnegative $A$
\begin{equation}
\nonumber \dint_0^A\,{\rm e}^{y^2}\,dy\;\leq\; \frac{{\rm e}^{A^2}}{A},
\end{equation}
we obtain (\ref{bound}) as desired.
\end{proof}

\begin{lem}
\label{l3} For any $a\geq 1$ and $k\geq 1$, denote by
$$
I(a,k)=\dint_{{\rm e}^{-\frac{k}{2}}}^1\, r {\rm e}^{\frac{4 a^2}{k}\log^2
r}\,dr.
$$
Then, we have
\begin{equation}
\label{Ia} I(a,k)\leq2 {\rm e}^{(a^2-1)k}.
\end{equation}
\end{lem}

\begin{proof}[Proof of Lemma \ref{l3}]
Changing the variable $u=-\frac{2a}{\sqrt k}\log r$ yields
$$
I(a,k)=\frac{\sqrt{k}}{2a}\,{\rm e}^{-\frac{k}{4a^2}}\,
\dint_{0}^{a\sqrt{k}}\,\,{\rm e}^{(u-\frac{\sqrt{k}}{2a})^2}\,du.
$$
Changing once more the variable $v=u-\frac{\sqrt k}{2a}$ yields

$$
I(a,k)=\frac{\sqrt{k}}{2a}\,{\rm e}^{-\frac{k}{4a^2}}\,
\dint_{-\frac{\sqrt{k}}{2a}}^{(2a^2-1)\frac{\sqrt{k}}{2a}}\,{\rm e}^{v^2}dv.
$$
Hence, for any $a\geq1$ we have

$$
I(a,k)\leq\frac{\sqrt{k}}{a}\,{\rm e}^{-\frac{k}{4a^2}}\,
\dint_{0}^{(2a^2-1)\frac{\sqrt{k}}{2a}}\,{\rm e}^{v^2}dv.
$$
Now, using the following estimate true for all nonnegative $A$
\begin{equation}
\nonumber \dint_0^A\,{\rm e}^{u^2}\,du\,\,\leq\frac{{\rm e}^{A^2}-1}{A}\leq
\frac{{\rm e}^{A^2}}{A},
\end{equation}
we obtain (\ref{Ia}) as desired.
\end{proof}

\begin{lem}
\label{Alamda} For any $\lambda>0$ and $A>\lambda$, denote by
$$
J(A,\lambda)=\dint_{A-\frac{\lambda^2}{A}}^{A}\,\frac{du}{\sqrt{{\rm e}^{A^2}-{\rm e}^{u^2}}}.
$$
Then, we have
\begin{equation}
\label{A-lamda} J(A,\lambda) \leq
\frac{A\,{\rm e}^{2\lambda^2}}{A^2-\lambda^2} \,{\rm e}^{-\frac{A^2}{2}}.
\end{equation}
\end{lem}

\begin{proof}[Proof of Lemma \ref{Alamda}]
Choosing $h(u)=\frac{-1}{u{\rm e}^{u^2}}$ and
$g'(u)=\frac{u{\rm e}^{u^2}}{\sqrt{{\rm e}^{A^2}-{\rm e}^{u^2}}}$, and integrating by
parts, we deduce (\ref{A-lamda}).\end{proof}

\begin{lem}
\label{A} For any $A>1$, denote by
$$
I(A)=\dint_{0}^A\,\frac{du}{\sqrt{{\rm e}^{A^2}-{\rm e}^{u^2}}}.
$$
Then, we have
\begin{equation}
\label{I(A)} I(A)\,\approx\,A\,{\rm e}^{-\frac{A^2}{2}}.
\end{equation}
\end{lem}
\begin{proof}[Proof of Lemma \ref{A}]
In one hand, it is clear that
$$
I(A)\geq\,A\,{\rm e}^{-\frac{A^2}{2}}.
$$

In the other hand, write
\begin{equation}
\label{integral}
 I(A)=
\dint_0^{A-\frac{1}{4A}}\,\frac{du}{\sqrt{{\rm e}^{A^2}-{\rm e}^{u^2}}}+J(A,\frac{1}{2}).
\end{equation}

By Lemma \ref{Alamda}, we get
$$
J(A,\frac{1}{2})\leq\,\frac{A\,{\rm e}^{\frac{1}{2}}}{A^2-\frac{1}{4}}
\,{\rm e}^{-\frac{A^2}{2}}\lesssim\,A\,{\rm e}^{-\frac{A^2}{2}}.
$$

Note that for any $0\leq u\leq A-\frac{1}{4A}$, we have

$$
\frac{1}{\sqrt{{\rm e}^{A^2}-{\rm e}^{u^2}}}\leq
\frac{1}{\sqrt{{\rm e}^{A^2}-{\rm e}^{(A-1/4A)^2}}}\lesssim\,{\rm e}^{-A^2/2}.
$$
Hence, the first integral in (\ref{integral}) can be estimated by
$$
\dint_0^{A-\frac{1}{4A}}\,\frac{du}{\sqrt{{\rm e}^{A^2}-{\rm e}^{u^2}}}
\lesssim\,A \,{\rm e}^{-A^2/2},
$$
and (\ref{I(A)}) follows.
\end{proof}

\section{Energy regularity data}
This section is devoted to the well-posedness issues in the energy space. Some of these results were announced in \cite{IMM2}. We begin by Theorem \ref{energy}.
\subsection{Proof of Theorem \ref{energy}}
First, consider the case $d\geq 3$.
\begin{enumerate}
\item[1)]
\noindent$\bullet$ {\it Construction of $\varphi_k$}.\\
For $k\geq 1$ and $\varepsilon>0$ (depending on $k$ as we will see
later) define the function $\varphi_k$ by
\begin{eqnarray*}
 \varphi_k(x)&=&\; \left\{
\begin{array}{cllll}0 \quad&\mbox{if}&\quad
|x|\geq 1,\\\\
a(k,\varepsilon)\Big(|x|^{2-d}-1\Big) \quad&\mbox{if}&\quad
\dfrac{\varepsilon}{k}\leq
|x|\leq 1,\\\\
k^{\frac{d-2}{2}}\quad&\mbox{if}&\quad |x|\leq
\dfrac{\varepsilon}{k},
\end{array}
\right.
\end{eqnarray*}
where $a(k,\varepsilon)$ is chosen such that $\varphi_k$ is
continuous, namely
$$a(k,\varepsilon)=\dfrac{\varepsilon^{d-2}k^{\frac{d-2}{2}}}{k^{d-2}-\varepsilon^{d-2}}.$$
An easy computation yields
$$
 \|\nabla\varphi_k\|_{L^2}^2\lesssim\;\dfrac{\varepsilon^{d-2}k^{d-2}}{k^{d-2}-\varepsilon^{d-2}}\lesssim\varepsilon^{d-2}.
 $$
Using assumption \eqref{3+}, we get
\begin{eqnarray*}
\dint_{\R^d}\;F(\varphi_k(x))\;dx&\lesssim&F(k^{\frac{d-2}{2}})\Big(\frac{\varepsilon}{k}\Big)^d+
\dint_{\frac{\varepsilon}{k}}^1
F\Big(a(k,\varepsilon)\Big(r^{2-d}-1\Big)\Big)\;r^{d-1}\;dr\\\\
&\lesssim&F(k^{\frac{d-2}{2}})\Big(\frac{\varepsilon}{k}\Big)^d\Big(1+\dfrac{
1-(\frac{\varepsilon}{k})^d}{(1-(\frac{\varepsilon}{k})^{d-2})^{\frac{2d}{d-2}}}\Big).
\end{eqnarray*}
Since $k\;\Big(F(k^{\frac{d-2}{2}})\Big)^{-\frac{1}{d}}\rightarrow
0$ we will choose
$$
\varepsilon=\varepsilon_k{\buildrel\hbox{\footnotesize def}\over
=}k\;\Big(F(k^{\frac{d-2}{2}})\Big)^{-\frac{1}{d}}.
$$
With this choice, we can see that
$\|\nabla\varphi_k\|_{L^2}\rightarrow 0$ and
$\dsup_k\;E(\varphi_k)<\infty$.\\
\noindent$\bullet$ {\it Construction of $t_k$}.\\
Consider the ordinary differential equation associated to
\eqref{wave1}.
\begin{equation}
\label{ODE-f}
\ddot{\Phi}+F'(\Phi)=0,\quad(\Phi(0),\dot{\Phi}(0))=(k^{\frac{d-2}{2}},0).
\end{equation}
Using Lemma \ref{l-ode} and the assumptions on $F$, we can see that
(\ref{ODE-f}) has a unique global periodic solution $\Phi_k$ with
period
\begin{eqnarray*}
T_k&=&2\sqrt{2}\;\dint_0^{k^{\frac{d-2}{2}}}\;\dfrac{d\Phi}{\sqrt{F(k^{\frac{d-2}{2}})-F(\Phi)}}\\\\
&=&2\sqrt{2}\dfrac{k^{\frac{d-2}{2}}}{\sqrt{F(k^{\frac{d-2}{2}})}}\;\dint_0^1\;\Big(1-\frac{F(v
k^{\frac{d-2}{2}})}{F(k^{\frac{d-2}{2}})}\Big)^{-1/2}\;dv.
\end{eqnarray*}
By assumption \eqref{3+}, we get
\begin{eqnarray*}
T_k&\leq&2\sqrt{2}\dfrac{k^{\frac{d-2}{2}}}{\sqrt{F(k^{\frac{d-2}{2}})}}\;\dint_0^1\;\Big(1-v^{\frac{2d}{d-2}}\Big)^{-1/2}\;dv\\\\
&\lesssim& k^{\frac{d-2}{2}}\Big(F(k^{\frac{d-2}{2}})\Big)^{-1/2}.
\end{eqnarray*}
It follows that
$$
T_k\ll\frac{\varepsilon_k}{k}.
$$
Now we are in a position to construct the sequence $(t_k)$. Recall
that by finite speed of propagation, any weak solution $u_k$ of
(\ref{wave1}) with data $(\varphi_k, 0)$ satisfy
$$
u_k(t,x)=\Phi_k(t)\quad\mbox{if}\quad
0<t<\frac{\varepsilon_k}{k}\;\;\mbox{and}\;\;|x|<\frac{\varepsilon_k}{k}-t.
$$
Hence
$$
|\partial_t
u_k(t,x)|=|\dot{\Phi}_k(t)|=\sqrt{2}\sqrt{F(k^{\frac{d-2}{2}})-F(\Phi_k(t))}.
$$
Let us choose $t_k=T_k/4$, then $\Phi_k(t_k)=0$, $t_k\ll\frac{\varepsilon_k}{k}$ and, for
$|x|<\frac{\varepsilon_k}{k}-t_k$,
$$
|\partial_t
u_k(t_k,x)|=\sqrt{2}\sqrt{F(k^{\frac{d-2}{2}})-F(\Phi_k(t_k))}\gtrsim\sqrt{F(k^{\frac{d-2}{2}})}.
$$
So
$$
 \|\partial_t u_k(t_k)\|_{L^2}^2\gtrsim\;F(k^{\frac{d-2}{2}})\Big(\frac{\varepsilon_k}{k}-t_k\Big)^d=
 \Big(\frac{\varepsilon_k}{k}\Big)^d\;F(k^{\frac{d-2}{2}})\Big(1-t_k\frac{k}{\varepsilon_k}\Big)^d,
$$
and the conclusion follows.

\item[2)] Now we turn to the proof of the second claim of Theorem \ref{energy}. For the sake of clearness, we restrict ourselves to the model example \eqref{3+}.
For any real $a>0$, we denote by $\Phi_a$ the unique global solution of
\begin{equation}
\label{ode}
\ddot{\Phi}(t)+\Phi^7(t)=0,\quad\left(\Phi(0),\dot{\Phi}(0)\right)=\left(a,0\right)\,.
\end{equation}
By Lemma \ref{l-ode}, $\Phi_a$ is periodic with period $T(a)$. Observe that by a scaling argument, we have $T(a)=a^{-3}T(1)$ and therefore
\begin{equation}
\label{period}
T(a)=C\,a^{-3},
\end{equation}
for some absolute positive constant $C$.
Let $(M_k)$ be a sequence of integers tending to infinity and such that
\begin{equation}
\label{firstsequence}
M_k=o\left(k^{1/6}\right),\quad k\to\infty\,.
\end{equation}
We denote by $(\eta_k)$ the unique sequence in $(0,\infty)$ satisfying
\begin{equation}
\label{etak}
4M_k=\dfrac{1}{1-(1-\eta_k)^3}\,.
\end{equation}
As a consequence of these choices, we obtain the following crucial identity
\begin{equation}
\label{crucial}
M_k\,T(\sqrt{k})=\left(M_k-\frac{1}{4}\right)\,T(\sqrt{k}(1-\eta_k))\,.
\end{equation}
A good choice of the sequence $(t_k)$ is then
\begin{equation}
\label{sequ}
t_k=M_k\,T(\sqrt{k})\,.
\end{equation}
Taking advantage of \eqref{period} and \eqref{firstsequence}, we get $t_k\ll\,k^{-4/3}$.\\
\noindent$\bullet${\sc Construction of $\phi_k$}\\ The idea is to take a function $\phi_k$ oscillating between $\sqrt{k}$ and $\sqrt{k}(1-\eta_k)$ a certain number of times. First we choose a sequence $(N_k)$ of  even integers tending to infinity and such that
\begin{equation}
\label{jump}
N_k\,\sim\,C\, k^{1/6}\,M_k^2,
\end{equation}
and we set $\alpha_k:=10\,t_k\, N_k\, k^{4/3}\,\sim\,C\,M_k^3$. We divide the radial interval $k^{-4/3}\leq r\leq (\alpha_k+1)k^{-4/3}$ into $N_k$ sub-intervals each of them has a length $10\,t_k$ and write
$$
[k^{-4/3}, (\alpha_k+1)k^{-4/3}]=\bigcup_{j=0}^{N_k-1}\,[a_k^{(j)},a_k^{(j+1)}],
$$
where $a_k^{(j)}=k^{-4/3}+10j\,t_k$. Now consider $\phi_k$ which is continuous and oscillates between $\sqrt{k}$ and $\sqrt{k}(1-\eta_k)$ as follows:
\begin{eqnarray*}
\phi_k(r)&=&\sqrt{k},\quad r\leq k^{-4/3},\\
\phi_k(r)&=&\sqrt{k}(1-\eta_k),\quad k^{-4/3}+t_k\leq r\leq k^{-4/3}+9t_k,\\
\phi_k(r)&=&\sqrt{k},\quad  k^{-4/3}+11t_k\leq r\leq k^{-4/3}+19t_k,\\
\phi_k(r)&=&\cdots,\\
\phi_k(r)&=&\sqrt{k},\quad  k^{-4/3}+ (10N_k-1)t_k\leq r\leq k^{-4/3}+(10N_k-1) t_k,\\
\phi_k(r)&=&\sqrt{k},\quad r\geq k^{-4/3}+10N_kt_k,
\end{eqnarray*}
and $\phi_k$ is affine in the remaining intervals. An easy computation shows that
\begin{equation}
\label{data}
\|\nabla\phi_k\|_{L^2}^2\,\lesssim
\,N_k\Big(\frac{\sqrt{k}\eta_k}{t_k}\Big)^2\,(k^{-4/3})^3\,t_k\,k^{4/3}\,\lesssim\,\frac{1}{M_k}\,.
\end{equation}
Moreover, using the finite speed of propagation and the fact that
$$
\Phi_{\sqrt{k}}(t_k)=\sqrt{k},\quad\Phi_{\sqrt{k}(1-\eta_k)}(t_k)=0,
$$
we conclude that any weak solution $u_k$ to \eqref{NL3+} with data $(\phi_k,0)$ satisfies
\begin{equation}
\label{instab}
\|\partial_t\,u_k(t_k)\|_{L^2}^2\,\gtrsim\,N_k k^4 \left(k^{-4/3}\right)^4 t_k k^{4/3}\,\,\gtrsim\,M_k^3\,.
\end{equation}
This finishes the proof for $d\geq 3$.\\
The case $d=2$ can be handled in a similar way. We have just to make a suitable choice of the initial data.\\
\noindent$\bullet$ {\it Construction of $\varphi_k$}.\\
For $k\geq 1$, we define $\varphi_k$ by
\begin{eqnarray*}
 \varphi_k(x)&=&\; \left\{
\begin{array}{cllll}0 \quad&\mbox{if}&\quad
|x|\geq 1,\\\\
\dfrac{-2\sqrt{k}}{\log\Big(F(\sqrt{k})\Big)}\;\log|x|
\quad&\mbox{if}&\quad \varepsilon_k\;{\rm e}^{-k/2}\leq
|x|\leq 1,\\\\
\sqrt{k}\quad&\mbox{if}&\quad |x|\leq\varepsilon_k\;{\rm e}^{-k/2},
\end{array}
\right.
\end{eqnarray*}
where $\varepsilon_k={\rm e}^{k/2}\;\Big(F(\sqrt{k})\Big)^{-1/2}$. Remark
that, by \eqref{2Dinfty}, we have  $\varepsilon_k\longrightarrow 0$. An
easy computation (using assumption \eqref{2Dinfty}) yields to
$$
 \|\nabla\varphi_k\|_{L^2}^2\lesssim\;\frac{-1}{\log\varepsilon_k},
 $$
and
 $$
 \dint_{\R^2}\;F(\varphi_k(x))\;dx\lesssim
 \varepsilon_k^2\;{\rm e}^{-k}\;F(\sqrt{k})+\dint_{\varepsilon_k
 {\rm e}^{-k/2}}^1\;r\;{\rm e}^{4\frac{\log^2 r}{(\log
 F(\sqrt{k}))^2}}\;dr{\buildrel\hbox{\footnotesize def}\over =}(I)+(II).
 $$
The choice of $\varepsilon_k$ implies that $(I)\lesssim 1$. For the
term $(II)$, we use Lemma \ref{lem-bound}.

\noindent$\bullet$ {\it Construction of $t_k$}.\\
As in higher dimensions, we consider the associated ordinary
differential equation with data $(\sqrt{k},0)$. This equation has a
unique global periodic solution with period
$$
T_k=2\sqrt{2}\;\dint_0^{\sqrt{k}}\;\dfrac{d\Phi}{\sqrt{F(\sqrt{k})-F(\Phi)}}.
$$
By assumption \eqref{2Dinfty}, we get $$ T_k\lesssim
\sqrt{k}\;\frac{1}{A}\;\dint_{0}^A\,\frac{du}{\sqrt{{\rm e}^{A^2}-{\rm e}^{u^2}}}$$
where $A=\sqrt{\log F(\sqrt{k})}$.
It follows from Lemma \ref{A} that $T_k\ll
\varepsilon_k\;{\rm e}^{-k/2}$. Now, arguing exactly in the same manner as
in higher dimension we finish the proof for $d=2$.
\end{enumerate}

\subsection{Proof of Theorem \ref{loc-large}}\quad\\
The idea here is to split the initial data in a small part in $H^1\times L^2$ and a
smooth one. First, we solve the IVP with smooth initial
data to obtain a local and bounded solution $v$.
Then, we consider the perturbed equation satisfied by
$w:=u-v$ and with small initial data. Notice that similar idea was used in \cite{GP, Germain, KPV, Planchon}. Now we come to the details.\\
\noindent{\sc Existence.} \\Given initial data $(u_0,u_1)$ in the energy
space $H^1\times L^2$, we decompose it as follows
\begin{eqnarray*}
(u_0,u_1)&=&(u_0,u_1)_{<n}+(u_0,u_1)_{>n}\\&:=&S_n(u_0,u_1)+(I-S_n)(u_0,u_1)
\end{eqnarray*}
 where $n$ is a (large) integer to be chosen later.
Remark
that
$$
(u_0,u_1)_{>n}\rightarrow 0\quad\mbox{in}\quad H^1\times
L^2\quad\mbox{as}\quad n\rightarrow\infty,
$$
and, for every $n$, $(u_0,u_1)_{<n}\in H^2\times H^1.$
First we consider the IVP with regular data
\begin{equation}
\label{regular-eq}\Box v +v+f(v)=0,\quad (v(0,x),\partial_t
v(0,x))=(u_0,u_1)_{<n},\quad f(v)=v\left({\rm e}^{4\pi v^2}-1\right)\,. \end{equation}
 It is known that (\ref{regular-eq}) is well-posed. More precisely, there exist a time $T_n=T(\|(u_0,u_1)_{<n}\|_{H^2\times H^1})>0$ and a unique solution $v$ to
 (\ref{regular-eq}) in $C_{T_n}(H^2)\cap C^1_{T_n}(H^1)$. Moreover, we can choose
 $T_n$ such that $\|v\|_{L^\infty_{T_n}(H^2)}\leq (\|(u_0)_{<n}\|_{H^2}+1)$.

Next we consider the perturbed IVP with small data
\begin{equation}
\label{pert-eq}\Box w+w+f(w+v)-f(v)=0,\quad (w(0,x),\partial_t
w(0,x))=(u_0,u_1)_{>n}.
\end{equation}
 We shall prove that (\ref{pert-eq}) has a local in time solution in
the space ${\mathcal E}_T:=C_T(H^1)\cap C^1_T(L^2)\cap
L^4_T(C^{\frac{1}{4}})$ for suitable time $T>0$. This will be
achieved by a standard fixed point argument. We denote by $w_\ell$ the
solution of the linear Klein-Gordon equation with data $(u_0,u_1)_{>n}$,
$$
\Box w_\ell+w_\ell=0,\quad (w_\ell(0,x),\partial_t w_\ell(0,x))=(u_0,u_1)_{>n}.
$$
For a positive time $T\leq T_n$ and a positive real number $\delta$, we denote by
${\mathcal E}_T(\delta)$ the closed ball in ${\mathcal E}_T$ of
radius $\delta$ and centered at the origin. On the ball ${\mathcal
E}_T(\delta)$, we define the map $\Phi$ by$$ \Phi: w\in {\mathcal
E}_T(\delta)\longmapsto \tilde{w}
$$
where
$$
\Box\tilde{w}+\tilde{w}+f(w+w_\ell+v)-f(v)=0,\qquad
\left(\tilde{w}(0,x),\partial_t\tilde{w}(0,x)\right)=\left(0,0\right).
$$
By energy and Strichartz estimates, we get
\begin{eqnarray*}
\|\Phi(w)\|_{{\mathcal
E}_T}&\lesssim&\|f(w+w_\ell+v)-f(v)\|_{L^1_T(L^2)}\\\\
&\lesssim&\|w+w_\ell\|_{L^\infty_T(L^2)}\;\Big\|{\rm e}^{C\|w+w_\ell+v\|_{\infty}^2}+{\rm e}^{C\|v\|_{\infty}^2}\Big\|_{L^1_T}
\end{eqnarray*}
It is clear that $$\Big\|{\rm e}^{C\|v\|_{\infty}^2}\Big\|_{L^1_T}\lesssim
T {\rm e}^{C(\|(u_0)_{<n}\|_{H^2}+1)^2}$$ On the other hand, using the logarithmic
inequality we infer
$$
{\rm e}^{C\|w+w_\ell+v\|_{\infty}^2}\lesssim
{\rm e}^{C\|(u_0)_{<n}\|_{H^2}^2}\Big(C+\frac{\|w+w_\ell\|_{C^{1/4}}}{\delta+\varepsilon}\Big)^{C(\delta+\varepsilon)^2},
$$
where $\varepsilon^2=\|w_0\|_{H^1}^2+\|w_1\|_{L^2}^2$. By H\"older
inequality in time we deduce
$$
\Big\|{\rm e}^{C\|w+w_\ell+v\|_{\infty}^2}\Big\|_{L^1_T}\lesssim
{\rm e}^{C\|(u_0)_{<n}\|_{H^2}^2}\;T^{1-\frac{\beta}{4}}\;\left(T^{1/4}+\delta+\varepsilon\right)^{\beta},
$$
where $\beta:=C(\delta+\varepsilon)^2<4$ for $\delta$ and
$\varepsilon$ small enough. Finally, we get
$$
\|\Phi(w)\|_{{\mathcal
E}_T}\lesssim
(\delta+\varepsilon)\;{\rm e}^{C\|(u_0)_{<n}\|_{H^2}^2}\;\Big(T+T^{1-\frac{\beta}{4}}\;\left(T^{1/4}+\delta+\varepsilon\right)^{\beta}\Big).
$$
>From this inequality it follows immediately that, if $T$ is small
enough then $\Phi$  maps ${\mathcal E}_T(\delta) $ into itself. To
prove that $\Phi$ is a contraction (at least for $T$ small), we
consider two elements $w_1$ and  $w_2$ in ${\mathcal E}_T(\delta)$
and define
$$ w=w_1-w_2,\quad\tilde{w}=\tilde{w_1}-\tilde{w_2},\quad
\bar{w}=(1-\theta)(w_\ell+w_1)+\theta(w_\ell+w_2)+v\quad\mbox{with}\quad
0\leq\theta\leq 1.
$$
We can write $$
f(w_\ell+w_1)-f(w_\ell+w_2)=w\big[(1+8\pi\bar{w}^2)
{\rm e}^{4\pi\bar{w}^2}-1\big]
$$ for some choice of $0\leq\theta(t,x)\leq 1$.
 By the energy estimate and the
Strichartz inequality we have
$$\Big\|\Phi(w_1)-\Phi(w_2)\Big\|_{{\mathcal
E}_T}\lesssim
\Big\|w {\rm e}^{C|\bar{w}|^2}\Big\|_{L^1_T(L^2_x)}.
$$
By convexity, we obtain
$$
\Big\|\Phi(w_1)-\Phi(w_2)\Big\|_{{\mathcal
E}_T}\lesssim\Big\|w\left(
{\rm e}^{C|w_\ell+w_1|^2}+{\rm e}^{C|w_\ell+w_2|^2}\right)\Big\|_{L^1_T(L^2_x)}.
$$
So arguing as before, we get
\begin{eqnarray*}
\Big\|\Phi(w_1)-\Phi(w_2)\Big\|_{{\mathcal
E}_T}&\lesssim&\|w\|_{L^{\infty}_T(L^2)}\;\left(\Big\|{\rm e}^{C\|w_\ell+w_1\|_{\infty}^2}\Big\|_{L^1_T}
+\Big\|{\rm e}^{C\|w_\ell+w_2\|_{\infty}^2}\Big\|_{L^1_T}\right),\\
&\lesssim&T^{1-\frac{\beta}{4}}\;\left(T^{1/4}+\delta+\varepsilon\right)^{\beta}\Big\|w_1-w_2\Big\|_T,
\end{eqnarray*}
for some $\beta<4$. If the parameters $\varepsilon>0$, $\delta>0$
and $T>0$ are suitably chosen, then $\Phi$ is a contraction map on ${\mathcal E}_T(\delta)$ and thus a local in time solution is constructed.\\
{\sc Uniqueness.}\\
We shall prove the uniqueness in the space
$$
{\mathcal F}_\eta:=C_{T}(H^2)\cap C^1_{T}(H^1) + \Big\{ \,w\in{\mathcal
E}_T\quad;\quad\|w\|_T\leq \eta\,\Big\},
$$
for any $\eta<\frac{1}{\sqrt{2}}$. Let $u:=v+w$ and $U:=V+W$ be two solutions of \eqref{nlkg} in ${\mathcal F}_\eta$ with the same initial data. Since $v, V\in C_t(H^2)$ and $H^2$ is embedded in $L^\infty$, we can choose a time $T>0$ such that (for some constant $C$)
\begin{equation}
\label{uniq}
\|v\|_{L^\infty([0,T], L^\infty)}\,\leq\, C\quad\mbox{and}\quad \|V\|_{L^\infty([0,T], L^\infty)}\,\leq\, C\,.
\end{equation}
The difference $U-u$ satisfies
$$
\Box(U-u)+U-u=f(v+w)-f(V+W),\quad \Big( (U-u),\partial_t(U-u)\Big)(t=0)=\left(0,0\right)\,.
$$

Using the energy estimate and Strichartz inequality, we get
\begin{eqnarray*}
\|U-u\|_{{\mathcal E}_T}&\lesssim&\|f(v+w)-f(V+W)\|_{L^1_T(L^2)}\\
&\lesssim&\|(U-u)\Big( U^2({\rm e}^{4\pi U^2}-1)+u^2({\rm e}^{4\pi u^2}-1)\Big)\|_{L^1_T(L^2)}\\
&\lesssim&\|U-u\|_{L^\infty_T(L^{\frac{2}{\varepsilon}})}\;\|U^2({\rm e}^{4\pi U^2}-1)+u^2({\rm e}^{4\pi u^2}-1)\|_{L^1_T(L^{\frac{2}{1-\varepsilon}})},
\end{eqnarray*}
where $\varepsilon>0$ to be chosen small enough. To conclude the proof of the uniqueness, we have to estimate the term $\|u^2({\rm e}^{4\pi u^2}-1)\|_{L^1_T(L^{\frac{2}{1-\varepsilon}})}$ for example. Observe that, for any $\beta>0$ and $a>1$,
\begin{equation}
\label{simple1}
x^2({\rm e}^{4\pi x^2}-1)\leq C_\beta\;\left({\rm e}^{4\pi(1+\beta) x^2}-1\right),
\end{equation}
and
\begin{equation}
\label{simple2}
(x+y)^2\leq \frac{a}{a-1}\,x^2+a\,y^2\,.
\end{equation}
Hence
$$
\|u^2({\rm e}^{4\pi u^2}-1)\|_{L^1_T(L^{\frac{2}{1-\varepsilon}})}\,\lesssim\, \dint_0^T\;\Big(\dint_{\R^2}\;\left({\rm e}^{8\pi\frac{1+\beta}{1-\varepsilon}u^2}-1\right)\,dx\Big)^{\frac{1-\varepsilon}{2}}\,\;dt\,.
$$
Moreover, using \eqref{simple2}, we can write
\begin{equation}
\label{simple3}
{\rm e}^{8\pi\frac{1+\beta}{1-\varepsilon}u^2}-1\leq \Big({\rm e}^{8\pi\frac{1+\beta}{1-\varepsilon}\frac{a}{a-1}v^2}-1\Big) +\Big({\rm e}^{8\pi\frac{1+\beta}{1-\varepsilon} a w^2}-1\Big)+\Big({\rm e}^{8\pi\frac{1+\beta}{1-\varepsilon}\frac{a}{a-1}v^2}-1\Big)\Big({\rm e}^{8\pi\frac{1+\beta}{1-\varepsilon} a w^2}-1\Big)\,.
\end{equation}
To estimate the first term in the RHS of \eqref{simple3}, we use \eqref{uniq}. For the second term, we observe that
$$
\sqrt{2}\,\eta\,\sqrt{\frac{(1+\beta)a}{1-\varepsilon}}\to \eta\sqrt{2}<1\quad\mbox{as}\quad a\to 1,\;\varepsilon,\,\beta\to 0\,.
$$
This enables us to use Moser-Trudinger inequality. We do the same for the last term. This concludes the proof of the uniqueness in the space ${\mathcal F}_\eta$. Note that we can weaken the hypothesis $\eta<\frac{1}{\sqrt{2}}$ to $\eta<1$ if we use the sharp logarithmic inequality \eqref{H-mu}.
\endproof

\begin{rem}
In higher dimension $d\geq 3$, we have a similar result in
$H^{d/2}\times H^{d/2-1}$ for (\ref{wave1}) by using a decomposition
in $H^{d/2+1}\times H^{d/2}$ and small in $H^{d/2}\times H^{d/2-1}$.
\end{rem}



\subsection{Proof of Theorem \ref{2Dsupercritical} }\quad\\
For any $k\geq 1$ define $f_k$ by:
\begin{eqnarray*}
 f_k(x)&=&\; \left\{
\begin{array}{cllll}0 \quad&\mbox{if}&\quad
|x|\geq 1,\\\\
-\dfrac{\log|x|}{\sqrt{k\pi}} \quad&\mbox{if}&\quad {\rm e}^{-k/2}\leq
|x|\leq 1
,\\\\
\sqrt{\frac{k}{4\pi}}\quad&\mbox{if}&\quad |x|\leq {\rm e}^{-k/2}.
\end{array}
\right.
\end{eqnarray*}
These functions were introduced in \cite{M} to show the optimality
of the exponent $4\pi$ in Moser-Trudinger inequality. An easy
computation shows that $\|\nabla
f_k\|_{L^2(\R^2)}=1$ and $\|f_k\|_{L^2(\R^2)}\lesssim\frac{1}{\sqrt{k}}.$
Denote by $u_k$ and $v_k$ any weak solutions
of (\ref{nlkg}) with initial data
$\Big((1+\frac{1}{k})f_k(\frac{\cdot}{\nu}), 0\Big)$ and $\Big(
f_k(\frac{\cdot}{\nu}), 0\Big)$ respectively. Observe that by
construction,
$$
 \|(u_k-v_k)(0)\|_{H^1}^2+
 \|\partial_t(u_k-v_k)0)\|_{L^2}^2=
 \frac{1}{k^2}\|f_k(\frac{\cdot}{\nu})\|_{H^1}^2=
\circ(1)\quad\mbox{as}\quad k\rightarrow\infty.
$$
Also, using estimate (\ref{Ia}), it is clear that
$$
0< E\Big((1+\frac{1}{k})f_k(\frac{\cdot}{\nu})\Big)-1\leq
{\rm e}^3\nu^2\quad\hbox{  and  }\quad 0<
E\Big(f_k(\frac{\cdot}{\nu})\Big)-1\leq \nu^2.
$$
Now, we shall construct the sequence of time $t_k$.
Observe that a good approximation of $u_k$ and $v_k$ is
provided by the corresponding ordinary differential equation,
\begin{eqnarray}
\label{ode expo}\ddot{\Phi}+\Phi {\rm e}^{4\pi\Phi^2}=0.
\end{eqnarray}
More precisely, let $\Phi_k$ and $\Psi_k$ be the solutions of
(\ref{ode expo}) with initial data

$$
 \Phi_k(0)=(1+\frac{1}{k})\sqrt{\frac{k}{4\pi}},
\quad\dot{\Phi}_k(0)=0,
$$
and

$$
 \Psi_k(0)=\sqrt{\frac{k}{4\pi}},
\quad\dot{\Psi}_k(0)=0,
$$
respectively.
Note that by finite speed of propagation, we have $\Phi_k=u_k$
and $\Psi_k=v_k$ in the backward light cone $$|x|<\nu
{\rm e}^{-k/2}-t,\quad t<\nu {\rm e}^{-k/2}.$$
 On the other hand, recall that
the period $T_k$ of $\Phi_k$ is given by
$$
T_k=4\int_0^{(1+\frac{1}{k})\sqrt{k}}
\frac{du}{\sqrt{{\rm e}^{(1+\frac{1}{k})^2 k}-{\rm e}^{u^2}}},
$$
hence, using Lemma \ref{A} we can prove that  $T_k \thickapprox
\sqrt{k}\, {\rm e}^{-(1+\frac{1}{k})^2k/2}$. Therefore, one need to
 choose time $t_k<<{\rm e}^{-(1+\frac{1}{k})^2k/2}$ and
 check that the decoherence of $\Phi_k$ and $\Psi_k$
 occurs at time $t_k$. Choose $t_k\in]0,T_k/4[$ such that
$$
\Phi_k(t_k)=(1+1/k)\sqrt{\frac{k}{4\pi}}-
\left((1+1/k)\sqrt{\frac{k}{4\pi}}\right)^{-1}.
$$
It follows that
$$
t_k=\dint_{\sqrt{k}+1/\sqrt{k}-
\frac{4\pi\sqrt{k}}{k+1}}^{\sqrt{k}+1/\sqrt{k}}\,\,
\frac{du}{\sqrt{{\rm e}^{k(1+1/k)^2}-{\rm e}^{u^2}}}.
$$
Using (\ref{A-lamda}), we obtain
$t_k\lesssim\frac{1}{\sqrt{k}}\,{\rm e}^{-k/2}$. In particular, if $k$ is
large enough then $t_k\lesssim\frac{\nu}{2}{\rm e}^{-k/2}.$
Now we show that this time $t_k$ is sufficient to let instability
occurs. Since $\Psi_k$ is decreasing on the interval $[0,\frac{T_k}{4}]$, we have
$$
{\rm e}^{4\pi\psi_k(0)^2}-{\rm e}^{4\pi\psi_k(t_k)^2}=|{\rm e}^k-{\rm e}^{4\pi\psi_k(t_k)^2}|\lesssim
{\rm e}^k,
$$
Therefore,
$$
|(\dot{\Phi}_k(t_k))^2-(\dot{\Psi}_k(t_k))^2|=\frac{1}{4\pi}\Big|\left({\rm e}^{4\pi\Phi_k(0)^2}-{\rm e}^{4\pi\Phi_k(t_k)^2}\right)-
\left({\rm e}^{4\pi\Psi_k(0)^2}-{\rm e}^{4\pi\Psi_k(t_k)^2}\right)\Big|\gtrsim{\rm e}^k.
$$

Finally, we deduce that
\begin{eqnarray*}
\dint_{\R^2}|\partial_t(u_k-v_k)(t_k)|^2\,dx&\gtrsim&
\dint_{|x|<\frac{\nu}{2}{\rm e}^{-k/2}}|\partial_t(u_k-v_k)(t_k)|^2\,dx\\\\
&\gtrsim&\nu^2 {\rm e}^{-k}|\dot{\Phi}_k(t_k))-\dot{\Psi}_k(t_k)|^2
\end{eqnarray*}
and the conclusion follows.\endproof


\section{Low regularity data}

\subsection{Proof of Theorem \ref{IP-2D}}
\begin{enumerate}
\item[1)] For $k\geq 1$ and $\gamma>1$, let $\phi_k=\gamma\,f_k$. An easy
computation shows that
$$
\|\nabla \phi_k\|_{L^{2,\infty}}\lesssim \frac{\gamma}{\sqrt{k}}.
$$
Next we consider the solution $\Phi_k$ of the associated O.D.E with
Cauchy data $\Big(\gamma\sqrt{\frac{k}{4\pi}},0\Big)$. The period
$T_k$ of $\Phi_k$ satisfies
$$
T_k\approx\gamma\sqrt{k}\,{\rm e}^{-\frac{\gamma^2}{2}k}\ll\,{\rm e}^{-\frac{k}{2}}.
$$
Arguing as in the previous section, we construct a sequence $(t_k)$
going to zero such that any weak solution $u_k$ with Cauchy data
$(\phi_k,0)$ satisfies
$$
\|\partial_t u_k(t_k)\|_{L^{2,\infty}}^2\gtrsim {\rm e}^{(\gamma^2-1)k},
$$
and we are done.
\item[2)] Now we will prove the ill-posedness in ${\mathcal B}^1_{2,\infty}$. The main difficulty is the construction of the initial data. For this end, consider a radial smooth function $h\in C_0^\infty(\R^2)$ satisfying $h(r)=0$ if $r\geq 2$ and $h(r)=1$ if $r<1$. For $a>0$, set $h_a(r)=h(\frac{r}{a})$. Since $\widehat{h_a}(\xi)=a^2\widehat{h}(a\xi)$, we get
\begin{equation}
\label{fourier1}
|\widehat{h_a}(\xi)|\leq\frac{C}{|\xi|^2}\quad\mbox{uniformly in}\quad a.
\end{equation}
Now we define the function $g_a$ via
$$
g_a(r)=\frac{1-h_a(r)}{r}.
$$
Our aim is to prove the following
\begin{prop}
\label{fourier2}
We have
$$
|\widehat{g_a}(\xi)|\leq\frac{C}{|\xi|}\quad\mbox{uniformly in}\quad a.
$$
\end{prop}
\begin{proof}
Write
$$
\widehat{g_a}(\xi)=\frac{C}{|\xi|}-C\left(\frac{1}{|\xi|}\star\widehat{h_a}(\xi)\right),
$$
where we have used the fact that $\widehat{r^{-1}}=C|\xi|^{-1}$. Observe that the convolution here is well defined.
Thus, we have to prove that, for fixed $\xi$,
$$
\Big|\int\,\frac{\widehat{h_a}(\eta)}{|\xi-\eta|}\,d\eta\Big|\lesssim \frac{1}{|\xi|}\quad\mbox{uniformly in}\quad a.
$$
The idea now is the following: fix $\xi$ such that $|\xi|\sim 2^j$ for some $j\in \Z$ and write
$$
\int\,\frac{\widehat{h_a}(\eta)}{|\xi-\eta|}\,d\eta=\int_{|\eta|\leq c2^j}\,\frac{\widehat{h_a}(\eta)}{|\xi-\eta|}\,d\eta+\int_{|\eta|\sim 2^j}\,\frac{\widehat{h_a}(\eta)}{|\xi-\eta|}\,d\eta+\int_{|\eta|\geq C2^j}\,\frac{\widehat{h_a}(\eta)}{|\xi-\eta|}\,d\eta.
$$
Using \eqref{fourier1}, we can easily estimate the second and the third term in the RHS. To estimate the first term, we use the fact that $\widehat{h_a}$ is uniformly in $L^1$.
\end{proof}
An immediate consequence is
\begin{cor}
\label{unifBesov}
We have
\begin{equation}
\label{unifBes}
\sup_{a>0}\,\|g_a\|_{\dot{\mathcal B}^0_{2,\infty}}<\infty\,.
\end{equation}
\end{cor}
\begin{proof}
Write
\begin{eqnarray*}
\|g_a\|_{\dot{\mathcal B}^0_{2,\infty}}&\approx&\sup_{j\in\Z}\,\dint_{2^{j-1}<|\xi|<2^{j+1}}\,|\widehat{g_a}(\xi)|^2\,d\xi,\\\\
&\lesssim&\;\sup_{j\in\Z}\,\dint_{2^{j-1}}^{2^{j+1}}\,\frac{dr}{r}\;\lesssim\; 1\quad\mbox{uniformly in}\quad a.
\end{eqnarray*}
\end{proof}
Now we are ready to construct the sequence of initial data $(\varphi_k)$. Let $\theta\in C_0^\infty(\R^2)$ be a radial function such that $\theta(r)=1$ if $r\leq 1$ and $\theta(r)=0$ if $r\geq 2$. For $k\geq 1$, set
\begin{equation}
 \label{data-besov}
\tilde{g}_k(r)=\frac{1}{\sqrt{k}}\;g_{{\rm e}^{-k/2}}\;(r)\;\theta(r)\,.
\end{equation}
It follows from Corollary \ref{unifBesov} that
$$
\|\tilde{g}_k\|_{\dot{\mathcal B}^0_{2,\infty}}\;\lesssim\;\frac{1}{\sqrt{k}}\;.
$$
Moreover, one can see easily that
$$
\frac{1}{C}\sqrt{k}\,\leq\, \dint_0^2\,\tilde{g}_k(r)\,dr\,C\sqrt{k}\,.
$$
To finish the construction set
$$
\varphi_k(r)=\gamma\,\sqrt{\frac{k}{4\pi}}-c_k\,\dint_0^r\,\tilde{g}_k(\tau)\,d\tau,
$$
where $\gamma>1$ and $c_k$ is chosen such that $\phi_k(2)=0$. The following proposition summarize some crucial properties of $\varphi_k$.
\begin{prop}
 \label{lim-Besov}
We have
\begin{itemize}
 \item[a)] $\varphi_k(r)=\gamma\,\sqrt{\frac{k}{4\pi}}$\quad if \quad $r\leq {\rm e}^{-k/2}$.
\item[b)] $\varphi_k\;\to\;0$\quad in\quad ${\mathcal B}^1_{2,\infty}(\R^2)$.
\end{itemize}
\end{prop}
\begin{proof}[Proof of Proposition \ref{lim-Besov}]
The first property follows directly from the definition of the function $\tilde{g}_k$. To prove the second, recall that
$$
\|\varphi_k\|_{{\mathcal
B}^1_{2,\infty}}\approx\|\varphi_k\|_{L^2}+\|\nabla\varphi_k\|_{\dot{\mathcal
B}^0_{2,\infty}}\,.
$$
Since $\|\varphi_k\|_{L^2}\;\lesssim\;\frac{1}{\sqrt{k}}$ we have just to prove that $\|\nabla\varphi_k\|_{\dot{\mathcal
B}^0_{2,\infty}}$ goes to zero. As $\nabla\varphi_k=\frac{x}{r}\;\tilde{g}_k(r)$, it suffices to apply Theorem \ref{Bourdaud} together with the fact that $\frac{x}{r}\in \dot{\mathcal B}^1_{2,\infty}\cap L^\infty$. This complete the proof of the proposition.
\end{proof}

Next, we consider the associated ODE with Cauchy data
$(\gamma\sqrt{\frac{k}{4\pi}},0)$ and denote by $\Phi_k$ the
 (global periodic) solution with period
$$
T_k\lesssim\dint_0^{\gamma\sqrt{k}} \frac{du}{\sqrt{{\rm e}^{\gamma^2
k}-{\rm e}^{u^2}}}\lesssim \gamma\sqrt{k}\,{\rm e}^{-\frac{\gamma^2}{2}k}\ll{\rm e}^{-\frac{k}{2}}\quad(\gamma>1).
$$
Set $t_k=T_k/4$ so that $\Phi_k(t_k)=0$.
Note that by finite speed of propagation any weak solution $u_k$ of
(\ref{Wave2D}) with Cauchy data $(\phi_k,0)$ satisfies
$$
u_k(t,x)=\Phi_k(t)\quad\mbox{for}\quad
0<t<{\rm e}^{-\frac{k}{2}}\quad\mbox{and}\quad |x|<{\rm e}^{-\frac{k}{2}}-t.
$$
Hence
\begin{equation}
\label{below}
-\partial_t u_k(t_k,x)\gtrsim
{\rm e}^{\frac{\gamma^2}{2}k}\quad\mbox{for}\quad
|x|<{\rm e}^{-\frac{k}{2}}-t_k.
\end{equation}
It remains to estimate from below the norm $\|\partial_t
u_k(t_k)\|_{\dot{\mathcal B}^0_{2,\infty}}$. To get the desired
estimate we proceed in the following way. First recall that
$$
\|\partial_t u_k(t_k)\|_{\dot{\mathcal
B}^0_{2,\infty}}=\sup_{\|v\|_{\dot{\mathcal
B}^0_{2,1}}=1}\,\dint_{\R^2}\,v(x)\,\partial_t u_k(t_k,x)\,dx.
$$
Then we have to make a suitable choice of $v$. Let $v$ be a smooth
compactly supported function such that
$$
v(x)=1\;\mbox{for}\; |x|\leq \frac{1}{4}\quad\mbox{and}\quad v(x)=0\;\mbox{for}\;|x|\geq\frac{1}{2}\,.
$$
For $k\geq 1$ let $v_k(x)={\rm e}^{\frac{k}{2}}\,v({\rm e}^{\frac{k}{2}} x).$
We remark that $\|v_k\|_{\dot{\mathcal
B}^0_{2,\infty}}=\|v\|_{\dot{\mathcal B}^0_{2,\infty}}$  is a
constant. Using \eqref{below}, we get
\begin{eqnarray*}
\|\partial_t u_k(t_k)\|_{\dot{\mathcal B}^0_{2,\infty}}&\geq&\dint\,-\partial_t u_k(t_k,x)\,v_k(x)\,dx\\
&\geq&{\rm e}^{\frac{k}{2}}\dint_{|x|\leq \frac{1}{4}{\rm e}^{-\frac{k}{2}}}\;-\partial_t u_k(t_k,x)\,dx\\&\gtrsim&
{\rm e}^{\frac{k}{2}}\,\left({\rm e}^{-\frac{k}{2}}\right)^2\,{\rm e}^{\frac{\gamma^2}{2}k}={\rm e}^{\frac{\gamma^2-1}{2}k}.
\end{eqnarray*}
This finishes the proof of the first part of the theorem since $\gamma>1$.


\item[2)] Without loss of generality, we may assume that $0\leq
s<1$. Let $0<\gamma<\frac{1}{2}(1-s)$ and consider
$\varphi_k=k^{\gamma}\,f_k$. It is clear that
$$
\|\varphi_k\|_{H^s}\lesssim k^{\gamma}\,k^{-\frac{1}{2}(1-s)}\to 0\quad (\gamma<\frac{1}{2}(1-s))\,.
$$
Denote by $u_k$ any weak solution of (\ref{nlkg}) with initial data
$(\varphi_k,0)$ and $\Phi_k$ the solution of the associated ODE
with Cauchy data $\left(k^{\gamma}\,\sqrt{\frac{k}{4\pi}},0\right)$.
The period $T_k$ of $\Phi_k$ satisfies
$$
T_k\lesssim k^{\gamma+\frac{1}{2}}\,{\rm e}^{-\frac{k^{2\gamma+1}}{2}}\ll
{\rm e}^{-\frac{k}{2}}.
$$
Choose $t_k=\frac{T_k}{4}$ so that $\Phi_k(t_k)=0$.
By finite speed of propagation, we have
$$
u_k(t,x)=\Phi_k(t),\quad |x|<{\rm e}^{-\frac{k}{2}}-t,\quad 0<t<{\rm e}^{-\frac{k}{2}}.
$$
Hence
$|x|<{\rm e}^{-\frac{k}{2}}-t_k$,
\begin{equation}
\label{below1}
-\partial_t u_k(t_k,x)\,=\,-\dot{\Phi}_k(t_k)=\frac{1}{2\sqrt{\pi}}\sqrt{{\rm e}^{k^{2\gamma+1}}-{\rm e}^{4\pi\Phi_k^2(t_k)}}=\frac{1}{2\sqrt{\pi}}{\rm e}^{\frac{1}{2}k^{2\gamma+1}}.
\end{equation}
To conclude the proof we need to estimate from below $\|\partial_t
u_k(t_k)\|_{H^{s-1}}$. Write
$$
\|\partial_t
u_k(t_k)\|_{H^{s-1}}=\sup_{\|v\|_{H^{1-s}}=1}\,\dint_{\R^2}\,v(x)\,\partial_t
u_k(t_k,x)\,dx.
$$
Let $v_k(x)={\rm e}^{\frac{sk}{2}}\,v({\rm e}^{\frac{k}{2}}\,x)$ where $v$ is as
above. It follows that
\begin{eqnarray*}
\|\partial_t u_k(t_k)\|_{H^{s-1}}&\geq&\dint\,-\partial_t u_k(t_k,x)\,v_k(x)\,dx\\&\geq&{\rm e}^{\frac{sk}{2}}\,\dint_{|x|\leq \frac{1}{4}{\rm e}^{-\frac{k}{2}}}\;-\partial_t u_k(t_k,x)\,dx\\&\gtrsim&
{\rm e}^{\frac{sk}{2}}\,({\rm e}^{-\frac{k}{2}})^2\,{\rm e}^{\frac{1}{2}k^{2\gamma+1}}={\rm e}^{(\frac{s}{2}-1)k+\frac{1}{2}k^{2\gamma+1}},
\end{eqnarray*}
which goes to infinity when $k\to\infty$.
\end{enumerate}

\subsection{Proof of Theorem \ref{WP-IP}}\quad\\
\begin{enumerate}
\item[1)] Our aim here is to prove the local well-posedness of equation \eqref{Wave2D} in the space ${\mathcal B}^1_{2,q'}\times{\mathcal B}^0_{2,q'}$ for any $1\leq q<\infty$. The strategy is the same as in the proof of Theorem \ref{loc-large}. We decompose the initial data $(u_0,u_1)$ into a small part\footnote{To do so in the case $q'=\infty$ we have to work with $\tilde{\mathcal B}^1_{2,\infty}:=\overline{\mathcal D}^{{\mathcal B}\,^1_{2,\infty}}$ and $\tilde{\mathcal B}^0_{2,\infty}:=\overline{\mathcal D}^{{\mathcal B}\,^0_{2,\infty}}$. }  in ${\mathcal B}^1_{2,q'}\times{\mathcal B}^0_{2,q'}$ and a regular one:
    $$
    (u_0,u_1)=(u_0,u_1)_{>N}+(u_0,u_1)_{<N}\,.
    $$
    First we solve the IVP with regular data to obtain a local regular solution $v$, and then we solve the perturbed IVP with small data using a fixed point argument to obtain finally the expected solution $u$. Let us start by studying the free equation.
For a given $(u_\ell^0,u_\ell^1)\in {\mathcal B}^1_{2,q'}\times{\mathcal B}^0_{2,q'}$ we denote by $u_\ell$ the free solution with data $(u_\ell^0,u_\ell^1)$, that is
    \begin{equation}
    \label{free}
    \Box u_\ell+u_\ell=0,\quad (u_\ell,\partial_t\,u_\ell)(t=0)=(u_\ell^0,u_\ell^1)\,.
    \end{equation}
 Using a localization in frequency, energy estimate and Strichartz inequality \eqref{Stz}, we derive the following result.

\begin{prop}
 \label{free-Bes}
 Let $T>0$. Then for any $1<q'\leq\infty$, there exists $0\leq\varepsilon(q')<1/4$ such that
 \begin{equation}
 \label{free-nader}
  \|u_\ell\|_{L^\infty_T({\mathcal B}^1_{2,q'})}+\|u_\ell\|_{L^{4}_T ({\mathcal C}^{\frac{1}{4}-\varepsilon})}\;\lesssim\;\|u_\ell^0\|_{{\mathcal B}^1_{2,q'}}+\|u_\ell^1\|_{{\mathcal B}^0_{2,q'}}\;.
 \end{equation}
 Actually, the proof shows that when $q'\leq 4$, we have a zero loss of derivatives meaning $\varepsilon(q')=0$, and if $q'>4$, one might choose an arbitrary $0<\varepsilon<1/4$.
 \end{prop}
 \begin{proof}[Proof of Proposition \ref{free-Bes}]\quad\\
\noindent
From the energy and Strichartz estimates applied to $\Delta_j  u_\ell$, we have
\begin{equation}
 \label{free2}
 2^j\|\Delta_j u_\ell\|_{L^\infty_T(L^2)}+2^{j/4}\|\Delta_j u_\ell\|_{L^{4}_T(L^\infty)}
\;\lesssim\;2^j\| \Delta_j  u_\ell^0\|_{L^2}+\|\Delta_j  u_\ell^1\|_{L^2}\,.
 \end{equation}
Summing the above estimate in $\ell^{q'}$ we have
\begin{eqnarray*}
\|2^{j/4}\|\Delta_ju_\ell\|_{L^4_T(L^\infty)}\|_{\ell^{q'}}\leq
\|u_\ell^0\|_{{\mathcal B}^1_{2,q'}}+\|u_\ell^1\|_{{\mathcal B}^0_{2,q'}}\,.
\end{eqnarray*}
In the case $q'\leq4$, the proposition follows from the observation
$$
\|u_\ell\|_{L^4({\mathcal B}^{1/4}_{\infty,q'})}\leq \|2^{j/4}\|\Delta_ju_\ell\|_{L^4_T(L^\infty)}\|_{\ell^{q'}}
$$
together with the Sobolev embedding ${\mathcal B}^{1/4}_{\infty,q'}\to {\mathcal C}^{1/4}$.
 When $q'>4$, notice that for any $0<\varepsilon<1/4$,
$$
\|u_\ell\|_{L^4_T({\mathcal B}^{1/4-\varepsilon}_{\infty,4})}=
\|\left(2^{j/4-j\varepsilon}\|\Delta_ju_\ell\|_{L^\infty}\right)_{\ell^4}\|_{L^4_T}=
\|2^{-j\varepsilon}\,\left(2^{j/4}\|\Delta_ju_\ell\|_{L^4_T(L^\infty)}\right)\|_{\ell^{4}}
$$
Using H\"older inequality in $j$ ($\frac{1}{4}=\frac{1}{q'}+\frac{1}{r}$ with $r=\frac{4q'}{q'-4}$) and \eqref{free2}, we get
\begin{eqnarray*}
\|u_\ell\|_{L^4_T({\mathcal B}^{1/4-\varepsilon}_{\infty,4})}&\leq& \|\left(2^{-j\varepsilon}\right)\|_{\ell^r}\,\|\left(2^{j/4}\|\Delta_ju_\ell\|_{L^4_T(L^\infty)}\right)\|_{\ell^{q'}}\\
&\lesssim& \|u_\ell^0\|_{{\mathcal B}^1_{2,q'}}+\|u_\ell^1\|_{{\mathcal B}^0_{2,q'}}\,.
\end{eqnarray*}
Again, Sobolev embedding enables us to finish the proof.
\end{proof}
Denote by $g_q(u):=u\Big((1+u^2)^{\frac{q-2}{2}}\;{\rm e}^{4\pi \left((1+u^2)^{\frac{q}{2}}-1\right)}-1\Big)$ so that the equation \eqref{Wave2D} reads
 \begin{equation}
 \label{Wave2d}
 \Box u+u+g_q(u)=0\,.
 \end{equation}
An easy computation shows that
{\tiny
\begin{equation}
\label{g_q}
|g_q(u)-g_q(v)|\,\leq\,\left\{
\begin{array}{cllll}C|u-v|\Big({\rm e}^{C|u|^q}-1+{\rm e}^{C|v|^q}-1\Big)\quad&\mbox{if}&\quad
1\leq q\leq 2,\\
C|u-v|\Big(u^2+{\rm e}^{C|u|^q}-1+v^2+{\rm e}^{C|v|^q}-1\Big) \quad&\mbox{if}&\quad 2<q<\infty.
\end{array}
\right.
\end{equation}}

 According to \eqref{g_q} and the Sobolev embeddings
\begin{eqnarray*}
H^1\hookrightarrow {\mathcal B}^1_{2,q'}\quad&\mbox{if}&\quad q\leq 2,\\
H^2\hookrightarrow {\mathcal B}^1_{2,q'}\hookrightarrow H^1\quad&\mbox{if}&\quad q> 2,
\end{eqnarray*}
 we will distinguish two cases.\\

\noindent$\bullet$\,\underline{\sf Case $1\leq q<2$}: \\

We solve $\Box v+v+g_q(v)=0$ with Cauchy data $(u_0,u_1)_{<N}\in H^1\times L^2$ to obtain a global solution $v\in {\mathcal C}(\R,H^1)$. Next we have to solve
\begin{equation}
\label{small}
\Box w+w+g_q(v+w)-g_q(v)=0,\quad (w,\partial_t w)(t=0)=(u_0,u_1)_{>N}\,.
\end{equation}
We seek $w$ in the form
$$
w=u_\ell+{\mathbf w},
$$
where $u_\ell$ is the free solution with Cauchy data $(u_0,u_1)_{>N}$. Hence ${\mathbf w}$ solves
\begin{equation}
\label{small1}
\Box{\mathbf w}+{\mathbf w}+g_q(v+u_\ell+{\mathbf w})-g_q(v)=0,\quad ({\mathbf w},\partial_t {\mathbf w})(t=0)=(0,0)\,.
\end{equation}
We rely on estimates for the linear part $u_\ell$ given by Lemma \ref{free-Bes} in order to choose appropriate functional spaces for which a fixed point argument can be performed. We introduce, for any nonnegative time $T$ and some $0\leq\varepsilon<1/4$, the following complete metric space
$$
{\mathcal E}_T={\mathcal C}([0,T],H^1(\R^2))\cap{\mathcal
C}^1([0,T], L^2(\R^2))\cap L^{4}_T({\mathcal C}^{\frac{1}{4}-\varepsilon}(\R^2)) $$
endowed with the norm
$$
 \|u\|_{{\mathcal E}_T}:=\sup_{0\leq t\leq
T}\Big[\|u(t,.)\|_{H^1}+\|\partial_tu(t,.)\|_{L^2}\Big]+\|u\|_{L^{4}_T({\mathcal
C}^{\frac{1}{4}-\varepsilon})}.
$$
 For a positive real number $\delta$, we
denote by ${\mathcal E}_T(\delta)$   the ball in ${\mathcal E}_T$
of radius $\delta$ and centered at the origin. On the ball
${\mathcal E}_T(\delta) $, we define the map $\Phi$ by
\begin{eqnarray}
\label{e9} {\mathbf w}\longmapsto\Phi({\mathbf w}):=\tilde{{\mathbf w}},
\end{eqnarray}
where
\begin{eqnarray}
\label{10}
\Box\tilde{{\mathbf w}}+\tilde{{\mathbf w}}=g_q(v)-g_q(v+u_\ell+{\mathbf w}),\quad
\left(\tilde{{\mathbf w}},\partial_t\tilde{{\mathbf w}}\right)(t=0)=(0,0)\,.
\end{eqnarray}
To show that, for small $T$ and $\delta$, $\Phi$ maps ${\mathcal
E}_T(\delta)$ into itself and it is a contraction, we use Lemma \ref{free-Bes} together with Lemma \ref{logBeso} and \eqref{g_q}. We skip the detail here and we refer to \cite{IMM1} for similar arguments.\\

\noindent$\bullet$\,\underline{\sf Case $2<q <\infty$}:\\

The method is almost the same as above except for the choice of the functional spaces. First we solve $\Box v+v+g_q(v)=0$ with Cauchy data $(u_0,u_1)_{<N}\in H^2\times H^1$ to obtain a local solution $v\in {\mathcal C}((-T,T),H^2)$. Remember that in this case, the nonlinearity is too strong to solve the Cauchy problem in $H^1\times L^2$ (see Theorem \ref{Mainresult1}). Next we have to solve
\begin{equation}
\label{small2}
\Box w+w+g_q(v+w)-g_q(v)=0,\quad (w,\partial_t w)(t=0)=(u_0,u_1)_{>N}\,.
\end{equation}
We seek $w$ in the form
$$
w=u_\ell+{\mathbf w},
$$
where $u_\ell$ is the free solution with Cauchy data $(u_0,u_1)_{>N}$. Hence ${\mathbf w}$ solves
\begin{equation}
\label{small3}
\Box{\mathbf w}+{\mathbf w}+g_q(v+u_\ell+{\mathbf w})-g_q(v)=0,\quad ({\mathbf w},\partial_t {\mathbf w})(t=0)=(0,0)\,.
\end{equation}
We introduce, for any nonnegative time $T$, the following complete metric space
$$
{\mathcal E}_T={\mathcal C}([0,T],H^2(\R^2))\cap{\mathcal
C}^1([0,T], H^1(\R^2))\cap L^4_T( {\mathcal C}^{1/4}(\R^2)) $$
endowed with the norm $$ \|u\|_{{\mathcal E}_T}:=\sup_{0\leq t\leq
T}\Big[\|u(t,.)\|_{H^2}+\|\partial_tu(t,.)\|_{H^1}\Big]+\|u\|_{L^4_T({\mathcal
C}^{1/4})}.
$$
We denote by ${\mathcal E}_T(\delta)$   the ball in ${\mathcal E}_T$
of radius $\delta$ and centered at the origin. On the ball
${\mathcal E}_T(\delta) $, we define the map $\Phi$ by
\begin{eqnarray}
\label{99} {\mathbf w}\longmapsto\Phi({\mathbf w}):=\tilde{{\mathbf w}},
\end{eqnarray}
where
\begin{eqnarray}
\label{100}
\Box\tilde{{\mathbf w}}+\tilde{{\mathbf w}}=g_q(v)-g_q(v+u_\ell+{\mathbf w}),\quad
\left(\tilde{{\mathbf w}},\partial_t\tilde{{\mathbf w}}\right)(t=0)=(0,0)\,.
\end{eqnarray}
Having in hands Lemmas \ref{free-Bes}-\ref{logBeso} and \eqref{g_q}, we proceed in a similar way as in the previous case (see also \cite{IMM1}) but now we need to be more careful since the source term has to be estimated in $L^1_T(H^1)$ instead of $L^1_T(L^2)$. We refer also to \cite{CIMM} for similar computation in the context of nonlinear Schr\"odinger equation.

\item[2)] Let us focus on the second part of the theorem. Without loss of generality, we may assume that $0\leq
s<1$. Also, for the sake of simplicity, we take $q=1$. Let $\gamma>\frac{1}{2}$ and, for $k\geq 1$, consider the
function $g_k$ defined by
\begin{eqnarray*}
 g_k(x)&=&\; \left\{
\begin{array}{cllll}\sqrt{k}\quad&\mbox{if}&\quad |x|\leq {\rm e}^{-k/2},\\\\
-\frac{\sqrt{k}}{\log
2}\,\log|x|+\left(\sqrt{k}-\frac{k^{\frac{3}{2}}}{2\log
2}\right)\quad&\mbox{if}&\quad {\rm e}^{-\frac{k}{2}}\leq |x|\leq
2{\rm e}^{-\frac{k}{2}},\\\\
0 \quad&\mbox{if}&\quad
|x|\geq 2{\rm e}^{-\frac{k}{2}}.
\end{array}
\right.
\end{eqnarray*}

Remark that
$$
\|k^\gamma\,g_k\|_{H^s}\lesssim
k^{\gamma-s+\frac{3}{2}}\,{\rm e}^{-(1-s)\,\frac{k}{2}}\longrightarrow
0\qquad (k\longrightarrow\infty).
$$

Denote by $\Phi_k$ the solution of the associated O.D.E with Cauchy
data $\left(k^{\gamma+\frac{1}{2}}, 0\right)$. The period $T_k$ of
$\Phi_k$ satisfies
$$
T_k\lesssim
k^{\gamma+\frac{1}{2}}\,{\rm e}^{-\frac{1}{2}k^{\gamma+\frac{1}{2}}}\ll
{\rm e}^{-\frac{k}{2}}.
$$
Choose $t_k=T_k/4$ so that $\Phi_k(t_k)=0$. By finite speed of
propagation, any weak solution $u_k$ of (\ref{Wave2D}) satisfies, for
$|x|<{\rm e}^{-\frac{k}{2}}-t_k$,
$$
-\partial_t
u_k(t_k,x)\,=\,\dot{\Phi}_k(t_k)=\frac{{\rm e}^{-2\pi}}{2\sqrt{\pi}}\sqrt{{\rm e}^{\sqrt{k^{2\gamma+1}+1}}-{\rm e}^{\sqrt{\Phi_k^2(t_k)+1}}}\,\gtrsim\,
{\rm e}^{\frac{1}{2}k^{\gamma+\frac{1}{2}}}.
$$

So arguing exactly as before, we get
$$
\|\partial_t u_k(t_k)\|_{H^{s-1}}\gtrsim
({\rm e}^{-\frac{k}{2}})^2\,{\rm e}^{\frac{sk}{2}}\,{\rm e}^{\frac{1}{2}k^{\gamma+\frac{1}{2}}}={\rm e}^{(\frac{s}{2}-1)k+\frac{1}{2}k^{\gamma+\frac{1}{2}}}.
$$
This concludes the proof once $\gamma>\frac{1}{2}$.
\end{enumerate}
\endproof




\end{document}